\documentclass[11pt,a4paper]{article}
\usepackage[english]{babel}
\usepackage[latin1]{inputenc}
\usepackage{fancyhdr}
\usepackage{amscd}
\usepackage{hyperref}
\usepackage{graphicx}
\usepackage{newlfont}
\usepackage{amssymb}
\usepackage{amsmath}
\usepackage{latexsym}
\usepackage{mathtools}
\usepackage{amsthm}
\usepackage{dsfont}
\usepackage{yfonts}
\usepackage{xcolor}
\usepackage{cancel} 
\usepackage{mathrsfs}
\newcommand{\R}{\mathbb{R}}
\newcommand{\C}{\mathbb{C}}
\newcommand{\Z}{\mathbb{Z}}
\newcommand{\N}{\mathbb{N}}
\newcommand{\Id}{\mathds{1}}

\newcommand{\cH}{{\cal H}}

\newtheorem{lemma}{Lemma}[section]
\newtheorem{theorem}[lemma]{Theorem}
\newtheorem{proposition}[lemma]{Proposition}
\newtheorem{notation}[lemma]{Notation}

\definecolor{darkgr}{rgb}{0.0, 0.62, 0.42}

\newcommand{\dbar}{d\hspace*{-0.08em}\bar{}\hspace*{0.1em}}
\newcommand{\omgen}{\omega}
\newcommand{\egen}{\mathrm{e}}
\numberwithin{equation}{section}

\theoremstyle{definition}

\newtheorem{definition}[lemma]{Definition}

\newtheorem{remark}[lemma]{Remark}

\title{\textbf{Massive Cantor families of periodic solutions of resonant Klein-Gordon equation on $\mathbb{S}^3$}}

\begin{document}

 \author{Diego Silimbani\footnote{
International School for Advanced Studies (SISSA), Via Bonomea 265, 34136, Trieste, Italy. 
 \textit{Email: }  \texttt{dsilimba@sissa.it}
 }}

\date{}

\maketitle

\begin{abstract}
We prove existence and multiplicity of Cantor families of small amplitude analytic in time periodic solutions of the completely resonant cubic nonlinear Klein-Gordon equation on $\mathbb{S}^3$ for an asymptotically full measure set of frequencies close to 1. The solutions are constructed by a Lyapunov-Schmidt decomposition and a Nash-Moser iterative scheme. We first find non-degenerate solutions of the Kernel equation. Then we solve the Range equation with a Nash-Moser iterative scheme to overcome small divisors problems.
\end{abstract}
\tableofcontents
\section{Introduction}
Motivated by the study of stability of anti-de Sitter space-time (AdS), the purpose of this paper is to prove existence and multiplicity of positive measure Cantor families of small amplitude time-periodic solutions of the nonlinear Klein Gordon equation on $\mathbb{S}^3$
\begin{equation}\label{mainequation}
-\partial_{tt}\phi +\Delta_{\mathbb{S}^3}\phi-\phi=\phi^3,
\end{equation}
where $\phi:\R\times\mathbb{S}^3\rightarrow \mathbb{R}$, and $\Delta_{\mathbb{S}^3}$ is the Laplace-Beltrami operator on the 3-dimensional sphere $\mathbb{S}^3$.\\
At linear level, all the solutions of $(-\partial_{tt}+\Delta_{\mathbb{S}^3}-\mathds{1})\phi=0$ are $2\pi$-time periodic, i.e. have frequency $\omega = 1 $,  
 since the eigenvalues of $-\Delta_{\mathbb{S}^3}+\mathds{1}$ are the squares of the natural numbers.
A classical approach to look for periodic solutions of \eqref{mainequation} 
with  frequency $\omega \sim 1 $, 
bifurcating from the kernel $\operatorname{ker}(-\partial_{tt}+\Delta_{\mathbb{S}^3}-\mathds{1})$, consists in using a Lyapunov-Schmidt decomposition. We call \eqref{mainequation} a completely resonant PDE since $\operatorname{ker}(-\partial_{tt}+\Delta_{\mathbb{S}^3}-\mathds{1})$ is infinite dimensional.
The reversibility of \eqref{mainequation} implies that we can look for solutions which are even in time.

Equation \eqref{mainequation} has been suggested in \cite{Bizon-Rostwo, Mali-Rostwo,Bizon-Craps-altri} as a toy model of spherically symmetric Einstein-scalar field equations close to the AdS space-time which is the maximally symmetric solution to the vacuum Einstein equations with universally negative cosmological constant $\text{Ric}(g)=-\Lambda g$.\\
AdS stability/instability properties are not yet understood nowadays and it seems they depend on conformal boundary conditions. For example, AdS is expected to be stable under dissipative boundary conditions, see \cite{Holzegel-Luk-Smulevici-Warnick}, while it has been conjectured that it is unstable under fully reflective boundary conditions by Dafermos and Holzegel in \cite{Dafermos-Holzegel} and Anderson \cite{Anderson2006}. Numerical simulations in \cite{Bizon-Rostwo} for spherically symmetric Einstein massless scalar field equations, seem to support AdS instability conjecture, against the formation of black holes for small perturbations. Such a phenomenon is rigorously proven for Einstein-massless-Vlasov systems with spherical symmetry by Moschidis in \cite{Moschi1, Moschi2} where it is shown the existence of a one-parameter family of inital data arbitrarily close to AdS whose time evolution generates black-hole region.
In \cite{Bizon-Rostwo} it is also suggested that some small initial data could lead to stable solutions of Einstein-massless scalar field equation, and in \cite{Mali-Rostwo} Maliborski and Rostworowski construct such solutions by formal power series which are also supported by numerical simulations. In the works \cite{Dias-Horowitz-Santos, Dias-Horowitz-Marolf-Santos}, the conjecture of existence of time periodic solutions (geons) has been extended to the vacuum Einstein equations.


\smallskip

Small amplitude time periodic spherically symmetric solutions of equation \eqref{mainequation} have been constructed by formal power series expansions in \cite{ch0} by Chatzikaleas.
The absence of "secular terms" in the power series expansions is obtained using the method of Maliborski and Rostworowski \cite{Mali-Rostwo}, developed for the Einstein-Klein-Gordon equation. However, the presence of \emph{small divisors} prevents the convergence of such power series. This difficulty looks analogous to the convergence problem of ``Linstedt series"  of quasi-periodic solutions in Celestial Mechanics, devised since Poincar\'{e} \cite{Poincare}, and successfully overcome during the last century by the celebrated KAM theory.

The currently  rigorous existence results  of small amplitude time-periodic solutions of \eqref{mainequation} are \cite{cha.smu} and \cite{berti2023time}. These works construct 
periodic solutions, either spherically symmetric  
or Hopf-plane waves, 
 whenever the time frequency $\omega$  belongs to the set of strongly Diophantine numbers
\begin{equation}\label{omegone}
\Omega_{\gamma}:=\left\lbrace \omega \in \left[ \frac{1}{2},2\right]\, :\, |\omega\ell-\omega_j|\geq \frac{\gamma}{\langle\ell\rangle},\,\, \forall \ell, j \in \N,\,\, \ell\neq \omega_j \right\rbrace, \quad \omega_j=j+1,\quad \gamma\in \left]0,\frac{1}{6}\right[,
\end{equation}
where $\langle \ell \rangle:=\max \lbrace 1, |\ell|\rbrace$ and $ \N := \{0,1,2, \dots \} $. The values $\omega_j=j+1$ in \eqref{omegone} are the square root of the eigenvalues of the operator $-\Delta_{\mathbb{S}^3}+\mathds{1}$ appearing in equation \eqref{mainequation}.
The set $\Omega_\gamma$ is uncountable and  accumulates to $\omega=1$, but it has $0$ measure, as shown in \cite{Bambusi-Paleari2001}. 

Looking for time-periodic solutions of \eqref{mainequation} with frequency $\omega\in \Omega_\gamma$ avoids small divisors phenomena since the inverse of the Klein-Gordon operator 
$ -\omega^2\partial_{tt}+\Delta_{\mathbb{S}^3}-\mathds{1} $ is bounded when restricted to the Range of $-\partial_{tt}+\Delta_{\mathbb{S}^3}-\mathds{1}$. Indeed it acts as a 
diagonal operator 
with eigenvalues  $\omega^2\ell^2-\omega_j^2$, and the condition $\omega\in\Omega_\gamma$ implies $|\omega^2\ell^2-\omega_j^2|\geq \gamma$ for any $ \ell \neq \omega_j  $.
In this way 
the Range equation can be solved by a simple contraction argument.

The goal of this paper is to prove 
existence and multiplicity of small amplitude time-periodic solutions of \eqref{mainequation} for a much larger set of frequencies, with actually positive measure, in particular it has {\it asymptotically full} measure at 1.
We refer to Theorem \ref{teoremone} for a precise statement. 

Enlarging the set of frequencies for which time-periodic solutions exist is 
relevant from both a physical and mathematical point of view. From a physical point of view, since this model can be thought as a first effective equation to understand stability near AdS space, the fact that time-periodic solutions are provided for a large set of frequencies suggests that stability regions are actually an observable phenomenon and not only an anomalous event, which could be thought when the result is given for only a 0-measure set of frequencies like $\Omega_\gamma$. 
From a mathematical point of view  Theorem \ref{teoremone} is one of the few existence results of 
periodic solutions of completely resonant Hamiltonian PDEs in high space dimension with small divisors phenomena.

\smallskip

In order to prove existence of solutions for a set of frequencies  with
asymptotically full measure at $ \omega = 1 $,  
we consider  Diophantine-type conditions 
 of type 
\begin{equation}\label{Meln0}
 |\omega 	\ell-\omega_j|\geq \frac{\gamma}{\langle \ell \rangle^\tau} \, , \quad
\forall \ell \neq \omega_j \, , \quad \text{for \ some} \ \tau > 1 \, , 
\end{equation}
which are weaker than \eqref{omegone} and give rise to small divisors. This implies  
that the inverse of the Klein-Gordon operator 
$ -\omega^2\partial_{tt}+\Delta_{\mathbb{S}^3}-\mathds{1} $ is {\it unbounded}, actually   
it loses $\tau-1$ derivatives since $|\omega^2\ell^2-\omega_j^2|\geq  \gamma\langle\ell\rangle^{1-\tau}$. 
Thus, a standard contraction arguments to solve the Range equation like in \cite{cha.smu, berti2023time} fails, and a more refined Nash-Moser type iteration is needed. 
Actually for the convergence of the iterative scheme further non-resonance Melnikov 
conditions are required: for a solution of amplitude $\varepsilon $, they  
have  the form 
\begin{equation}\label{Meln}
 \left|\omega(\varepsilon)\ell-\omega_j-\varepsilon\frac{{m(\varepsilon)}}{\omega_j}\right|\geq \frac{\gamma}{\langle \ell \rangle^{\tau}}
 \, , \quad \forall \ell \neq \omega_j\,, \,  \ell\geq \frac{1}{3\varepsilon}\, ,  \quad \omega (\varepsilon) = \sqrt{1+\varepsilon} \, , 
\end{equation}
where $ \varepsilon \mapsto m (\varepsilon) $  is a suitable $ C^1 $ function, 
 see \eqref{cantor} for the precise expression.
 We underline that the need to impose non-resonance conditions as  \eqref{Meln} 
 is the ultimate reason why the solutions \eqref{eq:soluzioni2.0} are 
 not analytic in $ \varepsilon  $. For finite dimensional systems 
 the lack of analyticity in $\varepsilon$ of lower dimensional tori 
 has been rigorously proved for instance in \cite{GallGent} (the periodic orbits we find are 1-dimensional invariant tori in the infinite dimensional phase space). \\

We postpone after the statement of Theorem \ref{teoremone} further 
comments and comparisons with 
\cite{cha.smu,berti2023time} as well as related literature. 
We now introduce the functional setting to state rigorously Theorem \ref{teoremone}. 

\subsection{Main Result}

We look for solutions which are spherically symmetric in space  according to the next definition.
  \begin{definition}[Spherically symmetric functions]\label{sfe}
  Consider on $\mathbb{S}^3$ the standard spherical coordinates
$$
 	\scalebox{0.9}{$(0, \pi) \times (0, \pi) \times (0, 2\pi) \ni (x, \theta, \varphi) \mapsto  (\cos(x), \sin(x) \cos(\theta), \sin(x) \sin(\theta) \cos(\varphi), \sin(x) \sin(\theta) \sin(\varphi))\,.$} $$
 	We say that $\phi: \mathbb{S}^3 \rightarrow \R$ is \emph{spherically symmetric} if 
 \begin{equation}\label{ident}
 \phi(x,\theta, \varphi) = u(x) \otimes 1_{\theta, \varphi}\,, \quad \forall ( x, \theta, \varphi) \in (0, \pi) \times (0, \pi) \times (0, 2\pi)\,, \quad u : (0, \pi) \rightarrow \C\,,
 \end{equation}
 	where $1_{\theta, \varphi}$ is the function identically equal to $1$ for any $(\theta, \varphi)$.
We say that $\phi: \R \times \mathbb{S}^3 \rightarrow \R$ is \emph{spherically symmetric} if $\phi(t, \cdot)$ is spherically symmetric for any $t \in \R$, and we  identify $\phi(t,\cdot)$ with the function $u(t,x)$ according to \eqref{ident}.
 \end{definition}
The above identification between $ u $ and $ \phi $ corresponds to  the unitary isomorphism between Hilbert spaces:
\begin{equation}\label{equispazi}
\cH^0_x:=L^2([0,\pi],\sin^2x\dbar x)\simeq \lbrace \phi \in L^2(\mathbb{S}^3,\dbar \mu)\,:\, \phi\, \text{ is spherically symmetric} \rbrace
\end{equation}
where  $\dbar x=\frac{2}{\pi} dx$ is normalized so that the constant function $u\equiv 1$ has $L^2-$norm equal to $ 1 $, and 
$ \dbar \mu $ is the normalized measure of $ \mathbb{S}^3 $.

The Laplace-Beltrami operator 
leaves invariant the subspace of spherically symmetric functions
$L^2([0,\pi],\sin^2x\dbar x)$, 
and, when restricted to this subspace, it possesses the orthonormal basis of eigenfunctions 
	\begin{equation}\label{def.ej}
	\egen_j(x) := \frac{\sin((j+1)x)}{\sin(x)} \quad \forall j \in \N \, , 
	\end{equation} 
	with eigenvalues $
	\omega_j^{2} $ 
	where $ \omega_j := j+1$ for any $ j \in \N .$
	
As a consequence,  
the subspace
  of real valued spherically symmetric functions  of the  Sobolev space  $H^r(\mathbb{S}^3, \dbar \mu)$ 
 can be 
identified with the Hilbert space
$$\mathcal{H}^r_x:=\left\lbrace u \in \cH^0_x\, :\, (-\Delta_{\mathbb{S}^{3}}+\mathds{1})^{\frac{r}{2}}u \in \cH^0_x \right\rbrace=\left\lbrace u(x)=\sum\limits_{j\in \N}u_j \egen_j(x)\, :\, 
\| u \|_{\mathcal{H}^r_x}^2 := \sum\limits_{j\in \N} u_j^2\omega_j^{2r} <\infty \right\rbrace$$
with inner product
$$
\langle u_1,u_2 \rangle_{\mathcal{H}^r_x}:=\langle (-\Delta_{\mathbb{S}^{3}}^{ss}+\mathds{1})^{r}u_1,u_2\rangle_{\cH^0_x}=\sum\limits_{j\in \N} u_{1,j}u_{2,j}\omega_j^{2r}.
$$
We look for spherically symmetric  time-periodic solutions of \eqref{mainequation} in the following spaces.
\begin{definition}\label{anal.sp}
Let $\sigma \geq 0,\, s,r\in \R$, we introduce the 
 Hilbert spaces of functions 
$$
X_{\sigma, s, r}:=\left\lbrace u(t,x)=\sum\limits_{\ell \in \Z} \exp(i\ell t)u_\ell(x)\, \bigg|\, u_{-\ell}=u_{\ell}\in \cH^r_x,\,  
\| u \|_{\sigma,s,r}^2 := 
\sum\limits_{\ell\in \Z}\exp(2\sigma|\ell|)\langle \ell \rangle^{2s}\|u_\ell\|^2_{\cH^r_x} <\infty \right\rbrace,
$$
endowed with the scalar product $\langle u, v \rangle_{\sigma,s,r}:=\sum\limits_{\ell\in \Z}\exp(2\sigma|\ell|)\langle \ell \rangle^{2s}\langle u_\ell, v_\ell \rangle_{\cH^r_x}$. In the following we use the notation
$$
X_{\sigma,s}:=X_{\sigma,s,2},\,\quad \|\cdot\|_{\sigma,s}:=\|\cdot \|_{X_{\sigma,s,2}}.
$$
In view of the condition $u_{-\ell}=u_{\ell}$, these functions are actually real-valued and even in time and admit the representation $u(t,x)= 2 \sum\limits_{\ell \in \N} \cos(\ell t) u_{\ell}(x)$.
\end{definition}
For $\sigma>0,\, s\geq 0$ these spaces consist of all  even in time periodic functions, taking values in the Sobolev space $\cH^r_x$ which admit a time-analytic extension in the complex strip $|\textrm{Im}(t)|<\sigma$ with trace function in the lines $|\textrm{Im}(t)|=\sigma$ which belongs to $H_t^s(\mathbb{T}, \cH^r_x)$. For  any $ \sigma \geq 0,\, s>\frac{1}{2},\, r> \frac{3}{2}$ these spaces are algebras with respect to the product of functions, in particular
\begin{equation}\label{algebra}
\|u_1 \cdot u_2\|_{\sigma,s,r}\leq C(s,r)\|u_1 \|_{\sigma,s,r}\|u_2\|_{\sigma,s,r}.
\end{equation}


\begin{theorem}\label{teoremone}
Let $s>\frac{1}{2}$, $\bar{\sigma}>0$ and $m\in \N.$ There exist 
$  \varepsilon_0:=\varepsilon_0(\bar{\sigma}, s, m)>0$ small enough, a Cantor-like set $B_\infty:=B_\infty(s,\bar{\sigma},m)\subseteq [0,\varepsilon_0]$ with asymptoticallly full measure at 0, namely
\begin{equation}\label{fullmeas}
\lim\limits_{\eta\to 0^+} \frac{|B_\infty\cap [0,\eta] |}{\eta}=1 \, ,
\end{equation}
and $m+1$ curves $u_\varepsilon^{(0)},\dots u_\varepsilon^{(m)}: [0,\varepsilon] \rightarrow X_{\bar{\sigma},s}$ of class $C^1$, 
of the form
\begin{equation}\label{eq:soluzioni2.0}
\begin{aligned}
& u_\varepsilon^{(j)}(t,x)=\varepsilon^{\frac{1}{2}}\sqrt{\frac{4\omega_j}{3}}\cos\left(t\omega_j 
\right)\egen_j(x)+r_\varepsilon^{(j)}\left(t,x\right) \, , \quad 
\left\|r_\varepsilon^{(j)}\right\|_{\frac{\bar{\sigma}}{2},s}\lesssim_{\sigma,s,m} \varepsilon^{\frac{3}{2}} \, , 
\end{aligned}
\end{equation}
for any $ j=0,\dots,m $,  such that 
for any $\varepsilon \in B_\infty$
$$ 
\tilde{u}_\varepsilon^{(j)}(t,x):=u_{\varepsilon}^{(j)}(\sqrt{1+\varepsilon} \,  t,x) \, , \quad
j=0,\dots,m \, , 
$$ 
are $\frac{2\pi}{\sqrt{1+\varepsilon}}-$periodic, 
analytic in time solutions of the Klein-Gordon equation \eqref{mainequation} on $\mathbb{S}^3$.
\end{theorem}
We now make some simple comments on Theorem \ref{teoremone}.
\begin{enumerate}
\item{The set of frequencies $\omega(\varepsilon)=\sqrt{1+\varepsilon}$ for which we find $m$-distinct $\frac{2\pi}{\omega(\varepsilon)}-$periodic solutions also has full-asymptotic measure at $\omega=1$ for any $m \in \N$, see Remark \ref{rem315}.}
\item{By triangular inequality if $m_1\neq m_2$, then $u_\varepsilon^{(m_1)}$ and $u_\varepsilon^{(m_2)}$ are distinct functions, and so $\tilde{u}_\varepsilon^{(m_1)}$ and $\tilde{u}_\varepsilon^{(m_2)}$ are different $\frac{2\pi}{\omega(\varepsilon)}$-periodic solutions of \eqref{mainequation} for $\varepsilon$ small enough.}
\item{The functions $u_\varepsilon^{(j)}$ in \eqref{eq:soluzioni2.0} (and so the solutions $\tilde{u}^{(j)}_\varepsilon$ of \eqref{mainequation} we construct) are actually $C^\infty$ also in the variable $x$, as follows by the bootstrap argument of Lemma \ref{classical}.}
\item{By performing the change of variable $U(t,x)=u(t,x)\sin(x)$, equation \eqref{mainequation} is equivalent to the completely resonant nonlinear wave equation
\begin{equation}\label{toro}
\begin{cases}\partial_{tt}U-\partial_{xx}U=\frac{U^3}{\sin^2 x},\quad x\in ]0,\pi[,\\
U(t,0)=U(t,\pi)=0.
\end{cases}
\end{equation}
For this reason Theorem \ref{teoremone} can also be regarded as a result of existence and multiplicity of time periodic solutions for a large set of frequencies for the completely resonant 1d wave equation with Dirichlet boundary conditions and singular nonlinearity $\frac{U^3}{\sin^2x}$.
Existence of time-periodic solutions for completely resonant 1d nonlinear wave equation of type $-\partial_{tt}U+\partial_{xx}U=f(x,U)$
were first proved in the case of periodic boundary conditions in $x$ for the case $f(x,u)=\pm U^3$ in \cite{Lid-Shulman} for badly approssimable irrational freuqencies $\omega$, and then extended in the case of Dirichlet Boundary conditions for a more general class of nonlinearities of type $f(x,u)=a(x)u^p +O(u^{p+1})$ in \cite{Bambusi-Paleari2001,Berti-BolleCMP,Berti-BolleNA}.\\
Special quasi-periodic solutions with two frequencies living in a $0-$measure set have been constructed in \cite{Procesi2005,Berti-Procesi2005} for completely resonant wave equations.\\
The first existence results of time-periodic solutions with frequencies in a large Cantor set for completely resonant 1d wave equations with Dirichlet boundary conditions and analytic nonlinearity were proved in \cite{Gentile-Mastropietro-Procesi,Berti-BolleCantor,BaldiBertiWave}. The small divisors problem is overcome in \cite{Gentile-Mastropietro-Procesi} by Lindstedt series expansions, while in \cite{Berti-BolleCantor} by a Nash-Moser iterative scheme. In \cite{ Berti-BolleNodea} the result is extended in the case when the non linear function $f$ has only finite regularity ($f\in C^k$). 
The work \cite{BaldiBertiVibr} proves existence of periodic solutions via a Nash-Moser scheme for a more general equation, which describes how waves propagates in nonhomogeneous media, with a forcing nonlinear term which ensures the existence of nondegenerate solutions of the Bifurcation Equation. 

None of the previous mentioned results implies the existence of periodic solutions of \eqref{toro} since the nonlinear term is singular, namely $\frac{1}{\sin^2(x)}\notin H^r(\mathbb{S}^1),\, \forall r\in \R$.\\
}
\end{enumerate}
A much richer literature concerning existence of periodic and quasi periodic solutions of nonresonant or partially resonant nonlinear wave/Klein Gordon equations is available. In these cases one uses the mass or the potential in order to impose suitable nonresonance conditions on the linear frequencies. In this cases the bifurcation equation will be finite-dimensional.

We quote the KAM results of Kuksin \cite{Kuksin}, Wayne \cite{Wayne1990} and Poschel \cite {Poschel1996,Poschel1996'} for 1d analytic wave equations with Dirichlet boundary conditions. For periodic boundary conditions Craig and Wayne \cite{Craig-Wayne1993} introduced the Lyapunov-Schmidt decomposition approach and showed existence of time-periodic solutions via a Nash-Moser iterative scheme, extended for also quasi-periodic solutions by Bourgain \cite{Bourgain1994}. 
Subsequently Chierchia-You \cite{Chierchia-You2000} managed to prove existence of quasi-periodic solutions for wave equations extending the KAM approach \a la Kuksin. 

In higher space dimension we quote the works  \cite{Bourgain1995,BertiBolle2010,Berti-Bolle-Procesi2010,Berti-Procesi2011} proving the existence of periodic solutions and \cite{Bourgain2005,Berti-Bolle2012,Berti-Corsi-Procesi2015,Grebert-Paturel2016,Berti-Bolle_book} for quasi-periodic solutions.

All the above works treated the nonresonant or partially resonant case (when the bifurcation equation is finite dimensional). Existence of quasiperiodc solutions for completely resonant wave equations are still poorly understood, the only existence results in this direction are \cite{Procesi2005,Berti-Procesi2005} which prove existence of quasiperiodic solutions with two frequencies. It is a very interesting 
open question to establish if equation \eqref{mainequation} admits quasi-periodic solutions.

\subsection{Ideas of the proof}\label{sec.ideas}
In order to prove Theorem \ref{teoremone} we perform a Lyapunov-Schmidt decomposition, and then solve the Range equation where "small divisors" appears by a Nash-Moser scheme.\\
For this purpose, it is convenient to rescale the size $u\longmapsto \varepsilon^{\frac{1}{2}}u$ and the period $t\mapsto \omega t$ of the solutions of \eqref{mainequation}, and look for $2\pi-$periodic spherically symmetric solutions (according to the Definition \ref{sfe}) of the equation
\begin{equation}\label{maineq}
\mathcal{L}_\omega u=\varepsilon u^3,
\end{equation}
where we used the notations
\begin{equation}\label{def.A}
\mathcal{L}_\omega := -\omega^2\partial_{tt} - A,\quad \quad A :=
-\Delta_{\mathbb{S}^3} + \Id \, .
\end{equation}
We perform a Lyapunov-Schmidt  decomposition of equation \eqref{maineq} by introducing
\begin{gather}
	\label{def.V}
	\begin{aligned}
	V :=	\ker (-\partial_{tt} - A)&=\Big\lbrace u(t,x)=\sum\limits_{j,\ell\in \N} {u}_{\ell, j}\cos (\ell t)\egen_j(x)\ :\ {u}_{\ell, j}=0,\, \forall \ell\neq \omgen_j \Big\rbrace\\
	&= \Big\lbrace v(t,x)=\sum\limits_{j \in \N} {v}_{j}\cos (\omega_j t)\egen_j(x) \Big\rbrace\,,
	\end{aligned}
	\\
	\label{def.W}
	\begin{aligned}
		W:=&
		\operatorname{Rg}(-\partial_{tt} -A)=\Big\lbrace u(t,x)=\sum\limits_{j,\ell\in \N}{u}_{\ell, j}\cos (\ell t)\egen_j(x) \ : \,{u}_{\ell, j}=0,\, \forall \ell= \omgen_j \Big\rbrace\,\\
	&	=\Big\lbrace w(t,x)=\sum\limits_{\ell\in \N} \cos(\ell t)w_\ell (x)\,\, :\,\,w_\ell \in \cH^0_x,\,\, \langle w_\ell , \egen_{\ell-1}\rangle_{\cH^0_x}=0\Big\rbrace.
		\end{aligned}
\end{gather}
Note that $W=V^{\perp}$ in $X_{\sigma, s},\, \forall \sigma, s$.
We denote by $\Pi_V$, $\Pi_W$, the orthogonal projectors on $V$ and $W$ respectively.
By denoting $v=\Pi_V u,\, w=\Pi_W u$, a function $u$ solves \eqref{maineq} if and only if $v,w$ solve the system
\begin{equation}\label{lyap}
\begin{cases}
(\omega^2-1) A v - 	\varepsilon\Pi_{V} (v +w)^3 = 0 \,,\\

\mathcal{L}_{\omega} w -  \varepsilon\Pi_W (v+w)^3 = 0\,.
\end{cases}
\end{equation}
By imposing the natural amplitude-to-frequency relation $\omega^2-1=\varepsilon$, namely 
\begin{equation}\label{amplfreq}
\omega=\omega(\varepsilon)=\sqrt{1+\varepsilon} \, ,
\end{equation}the system \eqref{lyap} becomes 
\begin{gather}
\label{v.eq}
A v -\Pi_{V} (v +w)^3 = 0 \,,\\
\label{w.eq}
\mathcal{L}_{\omega} w -  \varepsilon\Pi_W (v +w)^3 = 0\,.
\end{gather}
For the Bifurcation Equation \eqref{v.eq} we will first find explicit solutions of the "resonant system" $Av=\Pi_V v^3$ (which corresponds to \eqref{v.eq} in the case $w=0$). These solutions have the form $\bar{v}_m=\pm\sqrt{\frac{4\omega_m}{3}}\cos(\omega_mt)\egen_m(x),\,\, m\in \N$ and we prove their nondegeneracy. In this way one can apply an Implicit function argument and find a solution $v_m(w)$ of \eqref{v.eq} for any  $\|w\|_{\sigma,s}\leq \rho$.\\
Since $\bar{v}_m$ are one-mode functions, they are clearly analytic, namely $\bar{v}_m \in X_{\sigma,s}$ for any $ \sigma,s$.\\
Then, in order to show that the Range equation \eqref{w.eq} admits solutions when $v=v_m(w)$ we build a Nash-Moser Iterative scheme in order to deal with the loss of derivatives caused by $\mathcal{L}_\omega^{-1}=(-\omega^2\partial_{tt}+\Delta_{\mathbb{S}^3}-\mathds{1})^{-1}$.  We construct a sequence of approximate solutions $\lbrace
w_n \rbrace_{n\in \mathbb{N}}$ which converges in the analytic space $X_{\frac{\sigma}{2},s}$ to a solution of \eqref{w.eq}.
The scheme we use to prove the existence of a solution for the Range equation \eqref{w.eq} is rather general. It is based on the following properties:
\begin{itemize}
\item{Algebra estimates \eqref{algebra}: they are fundamental to control the nonlinear term and every term where a multiplication appears. For this reason we require a minimal regularity in time ($s>\frac{1}{2}$) and in space ($r>\frac{3}{2}$).}
\item{Smoothing estimates \eqref{smooth}: these estimates are a standard tool when working with scales of Banach spaces, which is our case for $X_{\sigma,s}$.}
\item{Invertibility with loss of derivatives of the Linearized Operator.\\
The whole Section \ref{invlin} is devoted to prove Proposition \ref{inversolinearizzato}.
In order to build iteratively a sequence of approximating solutions $\lbrace w_n\rbrace_{n\in \N}$
of the Range equation \eqref{w.eq} 
 one has to invert the linearized operator 
\begin{equation}\label{pertL}
\mathfrak{L}_{n+1}(\varepsilon,w_{n}):=-\omega^2\partial_{tt}+\Delta_{\mathbb{S}^3}-\mathds{1}-3\varepsilon P_{n+1}\Pi_W(v(w_n)+w_n)^2\cdot \, ,  
\end{equation}
obtained by linearizing \eqref{w.eq} and projecting it on the first $L_{n+1}\sim 2^{n+1}$ time frequencies.   In Section \ref{invlin} we analyze in depth this operator and show that $\mathfrak{L}_{n+1}(\varepsilon,w_n)^{-1}$ exists, and loses $\tau-1$ Sobolev derivatives, if the time frequency $\omega$ satisfies suitable first order Melnikov conditions.
First note that the unperturbed operator $\mathcal{L}_\omega$ in \eqref{def.A},  which is diagonal on the time-space basis $\lbrace \exp(i \ell t) \egen_j(x) \rbrace_{\ell \neq j+1}$, 
 loses $\tau-1$ Sobolev derivatives in time if the frequency $\omega$ satisfies Diophantine type conditions \eqref{Meln0}.

Then, to prove Proposition \ref{inversolinearizzato}, 
we expand functions $u(t,x)\in X_{\sigma,s}$ only in time-Fourier basis $\lbrace
\exp (i\ell t)\rbrace_{\ell \in \Z}$, and we split the operator $\mathfrak{L}_{n+1}(\varepsilon,w_n)$ into its diagonal part $D=\lbrace D_\ell \rbrace_{\ell \in \Z}$ and its 
off-diagonal part with respect to this basis.
Each $ D_\ell $ is an operator, acting on functions depending only on the  space variable $x$, of the form 
$$
D_\ell=\omega^2\ell^2+\Delta_{\mathbb{S}^3}-\mathds{1}-\varepsilon b_0(x) \, , 
$$ 
where $b_0 (x) $ is the mean in time of the nonlinear term $b(t,x)=3(v(w_n)(t,x)+w_n(t,x))^2$.\\
With Sturm-Liouville theory we prove that for $\varepsilon$ small enough the operator $-\Delta_{\mathbb{S}^3}+\mathds{1}+\varepsilon b_0 (x)  $ can be diagonalized and its eigenvalues are close to $\omega^2\ell^2-\omega_j^2-\varepsilon \overline{b_0}$ (see Lemma \ref{sturm}), where $\overline{b_0}$ denotes the mean in $x$ of the function $b_0 (x) $.
A key role is played by the  Sobolev Embedding Lemma \ref{sob} which states that Sobolev Spaces of spherically symmetric functions $\cH^r_x$ can be embedded into the standard Sobolev Spaces on the unit Circle $\mathbb{S}^1$ with flat metric $H^{r+1+\delta}(\mathbb{S}^1,dx)$ for any $\delta > 0 $ arbitrarily small. This embedding implies in particular the off-diagonal decay estimates of Lemma \ref{elements} 
for the matrix entries which represent the action of the multiplication operator for $ b_0 (x) $ 
 in the basis
 $\egen_j(x)  $ of $ \cH^0_x $.
With these estimates one sees that natural conditions to be asked for $\omega$ in order to 
prove the invertibility of $ D $
and that its inverse 
loses $\tau-1$ derivatives are the so called first order Melnikov conditions
$$
|\omega\ell-\omega_j|\geq \frac{\gamma}{\langle \ell \rangle^{\tau}} \, , 
\quad 
\left|\omega\ell-\omega_j-\varepsilon\frac{\overline{b_0}}{2\omega_j}\right|\geq \frac{\gamma}{\langle \ell \rangle^{\tau}} \, , \quad \forall \ell, j \, .  
$$ 
A further analysis of the small divisors enable to control the off-diagonal part of the operator  $ \mathfrak{L}_{n+1}(\varepsilon,w_{n}) $. In Lemma \ref{smalldiv} it is shown that for $\tau \in ]1,2[$ the product of two small divisors is larger than a constant if the singular sites are close enough, deducing suitable bounds for the off-diagonal operators, see Lemmas \ref{r1}, \ref{r2}, which enable to treat them as a perturbation of the main diagonal part $ D $. In conclusion the linearized operator $ \mathfrak{L}_{n+1}(\varepsilon,w_{n}) $ can be inverted by Neumann Series.}
\end{itemize}
The existence of solutions of the Range equation is then deduced in Section 	\ref{sec.w} by 
implementing a Nash-Moser iterative scheme. 
At last we prove 
in Proposition \ref{measure} that the set of amplitudes $ \varepsilon $ for which we find periodic solutions has asymptotically full measure at 0.

\smallskip

 Now we explain why, despite the fact that the Hopf-Plane waves solutions in \cite{cha.smu} of equation \eqref{mainequation} are nondegenerate, we don't know yet how to prove their existence for an asymptotically full measure set of frequencies.
 The missing ingredient is the analogous of the estimate \eqref{elements} which we are not able to verify since the eigenfunctions in Hopf coordinate have a much more difficult explicit form with respect to the spherically symmetric eigenfunctions $\egen_j(x)$.

\vspace{0.3cm}
\textbf{Acknowledgments.} I thank Massimiliano Berti and Beatrice Langella for useful discussions during the preparation of this work. Research supported by PRIN 2020 (2020XB3EFL001) ``Hamiltonian and dispersive PDEs".
\section{Solution of the Bifurcation Equation}\label{sec.bifurcation}
In this section we solve the bifurcation equation \eqref{v.eq} by the Implicit function theorem.\\
First we remind the product rule of the spherically symmetric eigenfunctions of $-\Delta_{\mathbb{S}^3}$ defined in \eqref{def.ej}:
\begin{equation}\label{prodotto}
\egen_j(x)\egen_k(x)=\sum\limits_{\ell=0}^{\min\lbrace j,k\rbrace} \egen_{|j-k|+2\ell}(x) \, .
\end{equation}
Note that the product $\egen_j\cdot\egen_k$ contains the frequencies from $|j-k|$ to $j+k$ with the same parity.\\
\begin{proposition}\label{soluzbif}
Let $m\in \N$, $\sigma\geq 0$, $s>\frac{1}{2}$. There exists $\rho :=\rho(m,\sigma,s)>0$ and a smooth solution
\begin{equation}\label{vmdm}
v_m: \mathcal{D}^{W}_{\sigma,s}(\rho)\longrightarrow V\cap X_{\sigma,s+2},\,\,\, \textrm{where}\,\,\,\mathcal{D}^{W}_{\sigma,s}(\rho):=\lbrace w\in W\cap X_{\sigma,s}\, :\, \|w\|_{\sigma,s}\leq \rho \rbrace
\end{equation}
of the bifurcation equation \eqref{v.eq}, satisfying the following properties:
\begin{itemize}
\item{$v_m(0)=\bar{v}_m$ where $\bar{v}_m:=\alpha_m\cos(\omega_m t)\egen_m(x)$ and $\alpha_m=\pm\sqrt{\frac{4\omega_m}{3}}$.}
\item{$v_m$ has bounded derivatives on $\mathcal{D}^{W}_{\sigma,s}(\rho)$, in particular $\exists R=R(\sigma, s,m)>0$ such that
\begin{equation}\label{stimev}
\max\limits_{j=0,1,2,3}\sup\limits_{w\in\mathcal{D}^{W}_{\sigma,s}(\rho)}\|D_w^jv_m(w)\|_{\sigma,s}  \leq R,
\end{equation}
where $\|D^j_wv_m(w)\|_{\sigma,s}:=\sup\limits_{\|h\|_{\sigma,s}\leq 1, h\in W}\|D^j_wv_m(w)[h^j]\|_{\sigma,s+2}$ is the operatorial norm as a $j-$linear continuous map.}
\end{itemize} 
\end{proposition}
\begin{proof}
Consider the following operator:
$$
\mathfrak{F}: \left( V\cap X_{\sigma, s+2} \right) \times \left( W \cap X_{\sigma,s}\right) \longrightarrow V\cap X_{\sigma, s},\quad \mathfrak{F}(v,w)=Av-\Pi_V(v+w)^3.
$$
By algebra properties \eqref{algebra} of $X_{\sigma, s}$, and boundedness of $A:X_{\sigma,s+2}\cap V\mapsto X_{\sigma,s}\cap V$, the map $\mathfrak{F}$ is analytic.\\
We shall prove first $\mathfrak{F}(\bar{v}_m,0)=0$, and then that $D_v\mathfrak{F}(\bar{v}_m,0)$ is invertible. Since $\cos^3(\omega_m t)=\frac{3}{4}\cos(\omega_m t)+\frac{1}{4}\cos (3\omega_m t)$, we obtain $\bar{v}_m^3=\left(\frac{3\alpha_m^3}{4}\cos(\omega_m t)+\frac{\alpha_m^3}{4}\cos(3\omega_m t)\right)\egen_m^3(x)$, recalling the definition \eqref{def.V} of $V$ we have
$$\Pi_{V}(\bar{v}_m^3)=\frac{3\alpha_m^3}{4}\langle \egen_m^3,\egen_m\rangle_{\cH^0_x} \cos(\omega_m t)\egen_m(x)+\frac{\alpha_m^3}{4}\langle \egen_m^3,\egen_{3\omega_m-1}\rangle_{\cH^0_x} \cos(3\omega_m t)\egen_{3\omega_m-1}(x).$$
The product rule \eqref{prodotto} implies that $\egen_m(x)^3\in \operatorname{span}\lbrace \egen_j(x)\rbrace_{j=0}^{3m}$, thus $\langle \egen_m^3,\egen_{3\omega_m-1}\rangle_{\cH^0_x}=0$ since $3\omega_m-1=3m+2>3m$, moreover by product rule \eqref{prodotto} and $\cH^0_x$-orthonormality of $\lbrace \egen_j \rbrace_{j\in\N}$, one has $\langle \egen_m^3,\egen_m \rangle_{\cH^0_x}=\left\|\egen_m^2\right\|^2_{\cH^0_x}=\omega_m$. In conclusion:
$$\mathfrak{F}(\bar{v}_m,0)=\left(\alpha_m \omega_m^2-\frac{3\omega_m}{4}\alpha_m^3\right) \cos(\omega_m t)\egen_m(x)=0,\quad \text{since} \, \alpha_m=\pm \sqrt{\frac{4\omega_m}{3}}.$$
Now we shall prove that $D_v\mathfrak{F}(\bar{v}_m,0)$ is invertible. We start by writing explicitly:
\begin{equation}\label{difff}
D_v\mathfrak{F}(\bar{v}_m,0)[h]=Ah-3\Pi_V\left( \bar{v}_m^2h \right).
\end{equation}
We check how the operator $D_v\mathfrak{F}(\bar{v}_m,0)$ acts with respect to the basis $\left\lbrace \cos(\omega_j t)\egen_j(x)\right\rbrace_{j\in \N}$ (which is an orthogonal basis for $V\cap X_{\sigma,s},\, \forall \sigma, s$), namely $\forall j \in \N$ we compute $D_v\mathfrak{F}(\bar{v}_m,0)[\cos(\omega_jt)\egen_j(x)]=\cos(\omega_jt)A\egen_j(x)-3\Pi_V(\bar{v}_m^2\cos(\omega_jt)\egen_j(x))$. In order to do this we start by writing:
$$\cos(\omega_m t)^2\cos(\omega_j t)=\frac{\cos(\omega_j t)}{2}+\frac{\cos(|2\omega_m-\omega_j|t)+\cos((2\omega_m+\omega_j)t)}{4}.$$
We shall now compute $\Pi_V(\cos(\omega_j t)\egen_m^2\egen_j)$, $ \Pi_V(\cos((2\omega_j+\omega_m)t)\egen_m^2\egen_j)$, $\Pi_V(\cos(|2\omega_m-\omega_j| t)\egen_m^2\egen_j)$:
\begin{itemize}
\item{$\Pi_V(\cos(\omega_j t)\egen_m^2\egen_j)=\langle \egen_m^2 \egen_j,\egen_j\rangle_{\cH^0_x}\cos(\omega_j t)\egen_j(x)$
and $\langle\egen_m^2 \egen_j,\egen_j\rangle_{\cH^0_x}=\langle \egen_m^2, \egen_j^2\rangle_{\cH^0_x}=\omega_{\min\lbrace j,m\rbrace}$ by product rule \eqref{prodotto} and the orthonormality of $\lbrace \egen_j \rbrace_{n\in \N}$.}
\item{$\Pi_V(\cos((2\omega_j+\omega_m)t)\egen_m^2\egen_j)=\langle \egen_m^2\egen_j, \egen_{2m+j+2}\rangle_{\cH^0_x}\cos((2\omega_j+\omega_m)t)\egen_{2m+j+2}=0$ because by product rule \eqref{prodotto} $\egen_m^2\egen_j\in \operatorname{span}\lbrace e_k\rbrace_{k=0}^{2m+j}$, and so $\langle\egen_m^2 \egen_j,\egen_{2m+j+2}\rangle_{\cH^0_x}=0$.}
\item{
In order to compute $\Pi_V\left(\cos(|2\omega_m-\omega_j|t)\egen_m^2 \egen_j\right)$ we consider two different cases:
\begin{enumerate}
\item{If $j\leq 2m\Rightarrow |2\omega_{m}-\omega_j|=\omega_{2m-j}$, then as in previous cases we have:\\ $\Pi_V\left(\cos((2\omega_m-\omega_j)t)\egen_m^2 \egen_j\right)=\langle \egen_m^2 \egen_j,\egen_{2m-j}\rangle_{\cH^0_x}\cos(\omega_{2m-j}t)\egen_m^2 \egen_j$.\\
We can rewrite $\langle \egen_m^2 \egen_j,\egen_{2m-j}\rangle_{\cH^0_x}=\langle \egen_m\egen_j, \egen_m\egen_{2m-j}\rangle_{\cH^0_x}$, and using product rule \eqref{prodotto} $\egen_m\egen_j=\sum\limits_{k=0}^{\min\lbrace m,j\rbrace}\egen_{|m-j|+2k},\quad\quad \egen_m\egen_{2m-j}=\sum\limits_{k=0}^{\min\lbrace m,2m-j\rbrace}\egen_{|m-j|+2k}$, it follows:
$$\langle \egen_m\egen_j, \egen_m\egen_{2m-j}\rangle_{\cH^0_x}=\sum\limits_{k=0}^{\min\lbrace m,j\rbrace}\sum\limits_{k'=0}^{\min\lbrace m,2m-j\rbrace}\delta_{k,k'}=\omega_{\min\lbrace j,m,  2m-j\rbrace}.$$}
\item{If $j\geq 2m+1$ then $|2\omega_m-\omega_j|=j-2m-1$, in the case $j=2m+1$ then we have $\Pi_V\left(\egen_m^2 \egen_j\right)=0$ cfr. \eqref{def.V}. If $j>2m+1$ we have $\Pi_V\left(\cos(|2\omega_m-\omega_j|t)\egen_m^2 \egen_j\right)=\cos((j-2m-1)t)\langle \egen_{j-2m-2},\egen_j\egen_m^2\rangle_{\cH^0_x}\egen_{j-2m-2}(x)=0$ because
by product rule \eqref{prodotto} $\egen_m^2(x)\egen_j(x)\in \operatorname{span}\lbrace \egen_k(x)\rbrace_{k=j-2m}^{j+2m}$, then $\langle \egen_{j-2m-2},\egen_j\egen_m^2\rangle_{\cH^0_x}=0 $}
\end{enumerate}
}
\end{itemize}
Now, since $\bar{v}_m=\alpha_n \cos(\omega_mt)\egen_m(x)$, with $\alpha_m=\pm\sqrt{\frac{4\omega_m}{3}}$ we obtain:
\begin{equation}\label{moltiplic}
\begin{aligned}
3\Pi_V\left( \bar{v}_m^2\cos(\omega_jt)\egen_j(x)\right)=\tilde{\beta}_{m,j}\cos(\omega_jt)\egen_j(x)+\gamma_{m,j}\cos(\omega_{2m-j}t)\egen_{2m-j}(x),\\
\text{where} \quad \tilde{\beta}_{m,j}=2\omega_m\omega_{\min\lbrace j,m\rbrace},\quad \gamma_{m,j}=\begin{cases} \omega_m\omega_{\min\lbrace j,2m-j\rbrace}, & \textit{if}\,\, j\leq 2m \\
0,& \textit{if}\, \, j>2m.
\end{cases}
\end{aligned}
\end{equation}
Plugging \eqref{moltiplic} in \eqref{difff} we deduce the following formula:
$$
D_v\mathfrak{F}(\bar{v}_m,0)[\cos(\omega_jt)\egen_j(x)]=\beta_{m,j}\cos(\omega_jt)\egen_j(x)+\gamma_{m,j}\cos(\omega_{2m-j}t)\egen_{2m-j}(x),\quad  \text{where} \,\,\, \beta_{m,j}=2\omega_j^2-\tilde{\beta}_{m,j}.
$$
It follows immediately that $D_v\mathfrak{F}(\bar{v}_m,0)$ admits the following invariant subspaces:
\begin{itemize}
\item{$\textrm{span}\lbrace \cos(\omega_m t)\egen_m(x)\rbrace$, with associated eigenvalue $\Lambda_m=-2\omega_m^2$,}
\item{$\forall j>2m$, $\textrm{span}\lbrace \cos(\omega_j t)\egen_j(x)\rbrace$, with associated eigenvalue $\Lambda_j=\omega_j^2-2\omega_m^2$,}
\item{$\forall 0\leq j \leq m-1$, $\textrm{span}\lbrace \cos(\omega_j t)\egen_{j}(x),\cos(\omega_{2m-j} t)\egen_{2m-j}(x)\rbrace$, with associated symmetric matrix $A_{m,j}:=\begin{pmatrix}
\omega_j^2-2\omega_m\omega_j & -\omega_m\omega_j\\
-\omega_m\omega_j  & \omega_{2m-j}^2-2\omega_m^2
\end{pmatrix}$.\\
A direct computation shows $\det (A_{m,j})=-\omega_j(\omega_m-\omega_j)^2(4\omega_m-\omega_j)<0$, it follows that the eigenvalues of $A_{m,j}$ are not 0.}
\end{itemize}
Hence $D_v\mathfrak{F}(\bar{v}_m,0)$ can be diagonalized by an orthogonal base, and all its eigenvalues are different from 0, moreover $\forall j>2m$ its eigenvalues $\Lambda_j$ satisfy $\frac{1}{2}\omega_{j}^2\leq |\Lambda_j| \leq \omega_j^2$.\\
This is enough to conclude that $D_v\mathfrak{F}(\bar{v}_m,0) \in GL(V\cap X_{\sigma,s+2},V\cap X_{\sigma,s})$ and so $\exists C_m=C(m,\sigma,s)>0$ such that $\|D_v\mathfrak{F}(\bar{v}_m,0)^{-1}h\|_{\sigma,s+2}\leq C_m\|h\|_{\sigma,s},\, \forall h \in V\cap X_{\sigma,s}$.\\
This shows that $\mathfrak{F}$ satisfies the assumptions of Implicit Function theorem near the point $(\bar{v}_m,0)$, and so the thesis follows immediately.
\end{proof}
\remark By the proof of Proposition \ref{soluzbif} one can see that the radius $\rho=\rho(m,\sigma,s)$ can be chosen uniformly in $\sigma,$ when $\sigma$ and $s$ vary respectively over finite intervals $[\sigma_\infty, \sigma_0]$, $[s_0,s_1]$.\\
Having an uniform radius will be useful in the next section since we will construct iteratively a sequence $w_n\in X_{\sigma_n,s}$, with decreasing analytic regularity $\sigma_n$, such that $\|w_n\|_{\sigma_n,s}\leq\rho$ $\forall n$, so that $v_m(w_n)$ will be well defined at each step.

\section{Solution of the Range equation}\label{sec.w}
In this section we fix $\bar{\sigma}>0,\, s>\frac{1}{2},\, \gamma\in\left]0,\frac{1}{6}\right[,\, \tau\in ]1,2[,\,m\in \N$ and we denote for simplicity the function $v(w):=v_m(w)$ defined in Proposition \ref{soluzbif} and the common radius of definition
\begin{equation}\label{rho}\rho:=\min\limits_{\sigma\in \Sigma,\, \tilde{s}\in S}\rho(\sigma, \tilde{s},m),\quad \quad \Sigma:= [0,\bar{\sigma}],\,\,\, S:= \left[s,s+\frac{2\tau(\tau-1)}{2-\tau}\right].
\end{equation}
We want to solve the Range equation
\begin{equation}\label{wrange}
\mathcal{L}_\omega w=\varepsilon\Pi_W\Gamma(w) \, , \quad \omega=\omega(\varepsilon)=\sqrt{1+\varepsilon} \, ,
\end{equation}
where
\begin{equation}\label{Gammone}
\Gamma(w):= f(v(w)+w) \, ,\quad
f(u):=u^3 \, . 
\end{equation}
This entire section is devoted to show the following theorem.
\begin{theorem}\label{mainthmw} For any $\bar{\sigma}>0,\, s>\frac{1}{2},\, \gamma\in\left]0,\frac{1}{6}\right[,\, \tau\in ]1,2[,\,m\in \N$, there exist $\varepsilon_0:=\varepsilon_0( \gamma, \tau,\bar{\sigma},s, m)>0$ small enough, $K_2:=K_2(\bar{\sigma}, s, m)>0$, and a function $\tilde{w}(\cdot) \in C^{1}([0,\varepsilon_0],W\cap X_{\frac{\bar{\sigma}}{2},s})$ satisfying
\begin{equation}\label{stimetilda}
\|\tilde{w}(\varepsilon)\|_{\frac{\bar{\sigma}}{2},s}\leq K_2\frac{\varepsilon}{\gamma},\quad \quad \|\partial_\varepsilon\tilde{w}(\varepsilon)\|_{\frac{\bar{\sigma}}{2},s}\leq K_2\gamma^{-1} \, ,\,\,\, \forall \varepsilon \in [0,\varepsilon_0] \, ,
\end{equation}which is a solution of  the Range equation \eqref{wrange} for any $ \varepsilon$ belonging to
\begin{equation}\label{cantor}
\begin{gathered}
B_\infty:=\bigg\lbrace \varepsilon\in [0,\varepsilon_0]\, :\, \left| \omega(\varepsilon)\ell -\omega_j-\varepsilon \frac{M(\tilde{w}(\varepsilon))}{2\omega_j}\right|\geq \frac{2\gamma}{(\ell+\omega_j)^\tau},\\
|\omega(\varepsilon)\ell-\omega_j|\geq \frac{2\gamma}{(\ell+\omega_j)^\tau},\, \forall \ell,j\in \N\,:\, \ell\geq \frac{1}{3\varepsilon},\, \ell\neq \omega_j \bigg\rbrace,\\
\textit{where}\quad   M(w):=\frac{1}{2\pi}\int\limits_{0}^{2\pi}\int\limits_0^\pi (\partial_u f)(w+v(w))(t,x) \dbar x dt \, .
\end{gathered}
\end{equation}
\end{theorem}
The proof of this theorem will be consequence of several lemmas and propositions. We introduce orthogonal subspaces of the Range $W$ defined in \eqref{def.W}. Given $L_0\in \N_0$ (fixed later), we define for every $n\in \N$  
\begin{equation}\label{rangesplitting}
\begin{aligned}
&W^{(n)}:=\left\lbrace w\in W\, :\, w=\sum\limits_{0\leq\ell\leq L_n}\cos(\ell t)w_\ell (x) \right\rbrace,\quad
\\ &W^{(n)\perp}:=\left\lbrace w\in W\, :\, w=\sum\limits_{\ell> L_n}\cos(\ell t)w_\ell (x) \right\rbrace,\quad L_n:=L_02^n \, ,
\end{aligned}
\end{equation}
and we denote as $P_n$ and $P_n^\perp$ the respective orthogonal projectors.
\remark\label{remarkemme} The function $M(w)$ defined in \eqref{cantor} is a smooth map in $X_{\sigma,s}$ for $\sigma >\frac{1}{2}$, $s>\frac{3}{2}$ with bounded derivatives on bounded sets since $f$ is analytic map by \eqref{algebra} and $v$ is smooth by Proposition \ref{soluzbif}.
\subsection{Fundamental Properties}
Here we write the 3 fundamental properties we use to solve the Range equation by means of an iterative Nash-Moser scheme:
\begin{lemma}[REGULARITY of $\Gamma$]\label{gamma}
For any $s>\frac{1}{2}$ and $\bar{\sigma}>0$, there exist $\rho:=\rho(\bar{\sigma},s)>0$ and $R':=R'(\bar{\sigma},s)$ such that $\Gamma(\cdot)\in C^2\left(\mathcal{D}_{\sigma,s}^W(\rho)\right)$ (where $\mathcal{D}_{\sigma,s}^W(\rho)$ was defined in \eqref{vmdm}) and $\forall \sigma\in [0,\bar{\sigma}]$
\begin{equation}\label{R'}
\max\limits_{j=0,1,2,3}\sup\limits_{w\in \mathcal{D}^W_{\sigma,s}(\rho)} \|D^j\Gamma(w)\|_{\sigma,s}\leq R'.
\end{equation}
Here we denoted for $j=1,2,3$ $\|D^j\Gamma(w)\|_{\sigma,s}=\sup\limits_{\|h\|_{\sigma,s}\leq 1,h\in W}\|D^{j}\Gamma(w)[h^j]\|_{\sigma,s}$ the operatorial norm of $D^j\Gamma(w)$ in $B^j(X_{\sigma,s},X_{\sigma,s})$ which consists of all continuous $j-$linear maps from $X_{\sigma,s}^j$ to $X_{\sigma,s}$.
\end{lemma}
\begin{proof}
This follows directly by algebra properties \eqref{algebra} since $\Gamma$ is a polynomial in $w$ and $v(w)$, moreover $v(\cdot)$ is smooth with bounded derivatives in $\mathcal{D}^W_{\sigma,s}(\rho)$ by Proposition \ref{soluzbif}.\\
We remark that $R'$ is uniform in $[0,\bar{\sigma}]$.
\end{proof}
\begin{lemma}[SMOOTHING ESTIMATES]
\begin{equation}\label{smooth}
\forall w \in W^{(n)\perp}\cap X_{\sigma,s},\, \forall 0\leq \sigma'\leq\sigma,\, \|w\|_{\sigma',s}\leq \exp \left(-L_n(\sigma-\sigma')\right)\|w\|_{\sigma,s}.
\end{equation}
\end{lemma}
\begin{proof}
In fact, a direct computation shows:
\begin{align*}
\|w\|^2_{\sigma',s}=\sum\limits_{|\ell|>L_n}\exp(2\sigma' |\ell|)\langle \ell \rangle^{2s}\|w_\ell\|^2_{\cH^2_x}=\sum\limits_{|\ell|>L_n}\exp(-2(\sigma-\sigma') |\ell|)\exp(2\sigma\ell)\langle \ell \rangle^{2s}\|w_\ell\|^2_{\cH^2_x}\\
\leq \exp(-2(\sigma-\sigma')L_n)\sum\limits_{|\ell|>L_n}\exp(2\sigma\ell)\langle \ell \rangle^{2s}\|w_\ell\|^2_{\cH^2_x}=\exp(-2(\sigma-\sigma')L_n)\|w\|^2_{\sigma,s}.
\end{align*}
\end{proof}
The most important property of Nash-Moser scheme is the invertibility of linearized operator.
\begin{proposition}[INVERTIBILITY OF THE LINEARIZED OPERATOR]\label{inversolinearizzato}
Let $\gamma \in \left]0,\frac{1}{6}\right[,\, \tau \in]1,2[$, $\bar{\sigma}>0$, $s>\frac{1}{2}$. Given $\sigma \in [0,\bar{\sigma}]$, assume $\|w\|_{\sigma,s+\frac{2\tau(\tau-1)}{2-\tau}}<\rho$. Then there exists $\varepsilon_0:=\varepsilon_0(\bar{\sigma},s, m, \gamma, \tau)>0 $ small enough, such that for any $ \varepsilon$ in
\begin{equation}\label{cantorAn}
\begin{gathered}
G_n(w):=\bigg\lbrace \varepsilon \in [0,\varepsilon_0]\, :\, \left| \omega(\varepsilon)\ell -\omega_j-\varepsilon \frac{M(w)}{2\omega_j}\right|> \frac{\gamma}{(\ell+\omega_j)^\tau},\\
|\omega(\varepsilon)\ell-\omega_j|> \frac{\gamma}{(\ell+\omega_j)^\tau},\, \forall \ell,j\in \N\,:\, \frac{1}{3\varepsilon}\leq\ell\leq L_n,\,\omega_j\leq 2L_n,\, \ell\neq \omega_j\,  \bigg\rbrace,
\end{gathered}
\end{equation}
the linear operator
\begin{equation}\label{linop}
\mathfrak{L}_n(\varepsilon,w):=\mathcal{L}_{\omega(\varepsilon)}-\varepsilon P_n\Pi_W D_w\Gamma(w)
\end{equation}
is invertible on $W^{(n)}$ and, setting  
$ K := 8\cdot 9^2 $, 
\begin{equation}\label{stimalinearizzato}
\left\|\mathfrak{L}_n(\varepsilon,w)^{-1}[h]\right\|_{\sigma,s}\leq \frac{K}{\gamma}L_n^{\tau-1}\|h\|_{\sigma,s} \, ,\,\,\, \forall h \in W^{(n)}\cap X_{\sigma,s} \, .
\end{equation}
\end{proposition}
The proof of this Proposition will be the whole content of Section \ref{invlin}.\\
We remark that the estimate \eqref{stimalinearizzato} holding for any $ n \in \N$ is essentially equivalent to the fact that the linearized operator $\mathfrak{L}(\varepsilon, w)^{-1}$ loses $\tau-1$ Sobolev derivatives in time.

In the following we will often omit to write the dependence of quantities on $ m \in \N $.

\subsection{Nash-Moser Scheme}
Let $\sigma_0:=\bar{\sigma}>0,\, \theta>0$ such that $\frac{\pi^2\theta}{6}<\frac{\bar{\sigma}}{2}$. Let $\sigma_{n+1}:=\sigma_n-\frac{\theta}{1+n^2}$, so that $\sigma_\infty:=\lim_{n\to \infty}\sigma_n>\frac{\bar{\sigma}}{2}$.
\begin{proposition}\label{mainpropw}There exist
$ L_0:=L_0(\gamma,\tau,\bar{\sigma}, s)>0$ and $ \varepsilon_0:=\varepsilon_0(\gamma,\tau, \bar{\sigma},s,L_0)>0$ such that for any $ n\geq 0$ there exists a solution $w_n:=w_n(\varepsilon)\in W^{(n)}$ of the equation
\begin{equation}\label{wenne}
\mathcal{L}_\omega w=\varepsilon P_n\Pi_W\Gamma(w)
\end{equation}
defined inductively for $\varepsilon\in \mathcal{G}_n\subseteq \mathcal{G}_{n-1}\subseteq \dots \subseteq \mathcal{G}_0:=[0,\varepsilon_0]$, where
\begin{equation}\label{Adritto}
\mathcal{G}_n:=\mathcal{G}_{n-1}\cap G_n(w_{n-1})\neq \emptyset.
\end{equation}
Furthermore $w_n \in C^1\left(\mathcal{G}_n, W^{(n)}\cap X_{\sigma_n,s}\right)$ satisfies for some $K_1,\,K_1'>0$, not depending on $n,\varepsilon,\gamma$
\begin{equation}\label{estiwn}
\|w_n(\varepsilon)\|_{\sigma_n,s}\leq K_1\frac{\varepsilon}{\gamma},\quad \quad \|\partial_\varepsilon w_n(\varepsilon)\|_{\sigma_n,s}\leq K_1',\,\, \forall \varepsilon \in \mathcal{G}_n.
\end{equation}
Moreover $w_n=\sum\limits_{i=0}^n h_i$, with $h_i \in C^1\left(\mathcal{G}_n,W^{(i)}\cap X_{\sigma_i,s}\right)$ satisfying for $\chi=\frac{3}{2}$, $\bar{\chi} \in ]1,\chi[ $ and some $ \bar{C}:=\bar{C}(\tau,\bar{\sigma},s,\bar{\chi})>0$ 
\begin{equation}\label{stimacca}
\|h_i\|_{\sigma_i,s}\leq \frac{\varepsilon}{\gamma}\exp(-\chi^i),\quad \quad\left\| \partial_\varepsilon h_i(\varepsilon)\right\|_{\sigma_i,s}\leq \bar{C} \exp(-\bar{\chi}^i),\quad \forall i=0,\dots, n, \,\,\forall \varepsilon \in \mathcal{G}_n.
\end{equation}
In addition for any $n \in \N$, there exists $\tilde{w}_n(\varepsilon)\in C^1\left([0,\varepsilon_0],W^{(n)}\cap X_{\sigma_n,s}\right)$ such that:
\begin{enumerate}
\item\label{tilde1}{$\tilde{w}_n(\varepsilon)=w_n(\varepsilon),$ for any $  \varepsilon \in \bigcap\limits_{i=0}^n \tilde{\mathcal{G}}_i$, where $\tilde{\mathcal{G}}_i$ are defined for some $\nu>0$ as
\begin{equation}\label{Atildati}
\tilde{\mathcal{G}}_0:=[0,\varepsilon_0],\quad
\tilde{\mathcal{G}}_i:=\left\lbrace \varepsilon\in \mathcal{G}_i\,:\, dist(\varepsilon,\partial \mathcal{G}_i)\geq \frac{2\nu}{L_i^3}\right\rbrace,\,\, \forall i\geq 1.
\end{equation}}
\item\label{tilde2}{$\|\tilde{w}_n(\varepsilon)\|_{\sigma_n,s}\leq K_2\frac{\varepsilon}{\gamma},  \quad \|\partial_\varepsilon \tilde{w}_n(\varepsilon)\|_{\sigma_n,s}\leq K_2'\nu^{-1}$, for any $ \varepsilon \in [0,\varepsilon_0]$, for some constants $K_2,\, K_2'>0$ which do not depend on $n,\varepsilon, \gamma$.}
\end{enumerate}
\end{proposition}
The sequence $w_n$ will be constructed inductively.

\begin{lemma}[Initialization]
Let $L_0 \in \N$, then there exists $ \varepsilon_0:=\varepsilon_0(L_0)>0$ small enough such that for any $ \varepsilon \in [0,\varepsilon_0]$ the equation \eqref{wenne} for $n=0$,
\begin{equation}\label{w0}
\mathcal{L}_\omega h_0=\varepsilon P_0\Pi_W \Gamma(h_0),
\end{equation}
 admits a solution $h_0(\varepsilon)$ satisfying, for some $K_1,\,K_1'>0$, 
 $$
 \|h_0(\varepsilon)\|_{\sigma_0,s}\leq K_1 \frac{\varepsilon}{\gamma},\quad \|\partial_\varepsilon h_0(\varepsilon)\|_{\sigma_0,s}\leq K_1'.
$$
\end{lemma}
\begin{proof}
The first step is to prove that if $\frac{\varepsilon_0}{\omega+1}L_0\leq \frac{1}{2}$ then $|\omega^2\ell^2-\omega_j^2| \geq \frac{1}{2},\,\, \forall \ell\neq \omega_j$. Indeed $\forall 0\leq \ell \leq L_0$
$$|\omega^2\ell^2-\omega_j^2|=\underbrace{|\omega\ell+\omega_j|}_{\geq 1}|\ell-\omega_j +(\omega-1)\ell|\geq \underbrace{|\ell-\omega_j|}_{\geq 1}-\underbrace{|\omega-1|}_{=\frac{\varepsilon}{\omega+1}}\underbrace{|\ell|}_{\leq L_0}\geq 1-\frac{1}{2}.$$
It follows that $\mathcal{L}_\omega^{-1}$ is bounded in $W^{(0)}\cap X_{\sigma_0,s}$ and has norm $\leq 2$. In order to solve \eqref{w0} we show that, provided $|\varepsilon|\leq \varepsilon_0$ small enough, the map
$\mathcal{T}_\varepsilon^{(0)}(w):= \varepsilon \mathcal{L}_\omega^{-1}P_0\Pi_{W}\Gamma(w)$
 is a contraction on $\mathcal{D}_0:=\lbrace
w \in W^{(0)}\,:\, \|w\|_{\sigma_0,s}\leq \rho_0 \rbrace$, where $\rho_0=\frac{\rho}{2}$. Indeed by Lemma \ref{gamma}
\begin{itemize}
\item{$\left\|\mathcal{T}_\varepsilon^{(0)}(w)\right\|_{\sigma_0,s}\leq 2\varepsilon R'.$\\
Hence for $\varepsilon_0\leq \frac{\rho_0}{2 R'}$ one has that $\mathcal{T}_\varepsilon^{(0)}$ maps $\mathcal{D}_0$ into itself.}
\item{$\left\|D_w\mathcal{T}_\varepsilon^{(0)}(w)[h]\right\|_{\sigma_0,s}\leq 2\varepsilon R'\|h\|_{\sigma_0,s}.$\\
Thus, for $\varepsilon_0 \leq \frac{1}{4 R'}$ one has $\left\|D_w\mathcal{T}_{\varepsilon}^{(0)}(w)[h]\right\|_{\sigma_0,s}\leq\frac{1}{2}\|h\|_{\sigma_0,s}$, and so $\mathcal{T}_{\varepsilon}^{(0)}$ is a contraction.}
\end{itemize}
It follows that $\exists !\, h_0(\varepsilon)$ fixed point of $\mathcal{T}_{\varepsilon}^{(0)}$, moreover
$$\|h_0(\varepsilon)\|_{\sigma_0,s}=\|\mathcal{T}_{\varepsilon}^{(0)}(h_0(\varepsilon))\|_{\sigma_0,s}\leq 2 \varepsilon R'.$$
Using that $h_0(\varepsilon)$ is solution of \eqref{w0} we obtain for the derivative $
\|\partial_\varepsilon h_0(\varepsilon)\|_{\sigma_0,s}\leq 3 R'$.
\end{proof}
Suppose now that we already defined a solution $w_n\in W^{(n)}$ of the equation \eqref{wenne} satisfying the properties of Proposition \ref{mainpropw}. 
We want to define a solution $w_{n+1}=w_n+h_{n+1}$, with $h_{n+1}\in W^{(n+1)}$ of 
\begin{equation}\label{rangen+1}
\mathcal{L}_{\omega}(w_n+h_{n+1})-\varepsilon P_{n+1}\Pi_W(\Gamma(w_n+h_{n+1}))=0.
\end{equation}
Since by induction  $w_n(\varepsilon)\in W^{(n)}$ solves \eqref{wenne} for any $ \varepsilon \in \mathcal{G}_n$ we have
\begin{align*}
&\mathcal{L}_{\omega}(w_n+h)-\varepsilon P_{n+1}\Pi_W\Gamma(w_n+h)=\mathcal{L}_{\omega}h+\varepsilon P_{n}\Pi_W\Gamma(w_n)-\varepsilon P_{n+1}\Pi_W\Gamma(w_n+h)\\
&=\mathcal{L}_{\omega}h-\varepsilon P_{n+1}\Pi_W(\Gamma(w_n+h)-\Gamma(w_n))-\varepsilon (P_{n+1}-P_n)\Pi_W\Gamma(w_n)\\
&=\underbrace{\mathcal{L}_{\omega}h-\varepsilon P_{n+1}\Pi_WD_w\Gamma(w_n)[h]}_{=\mathfrak{L}_{n+1}(\varepsilon,w_n)[h]}-\varepsilon \mathcal{R}_n(h)-\varepsilon r_{n},
\end{align*}
where
\begin{equation}\label{resto1}
\mathcal{R}_n(h):= P_{n+1}\Pi_W\left(\Gamma(w_n+h)-\Gamma(w_n)-D_w\Gamma(w_n)[h]\right),
\end{equation}
\begin{equation}\label{resto2}
r_{n}:= P_{n+1}P_n^{\perp}\Pi_W\Gamma(w_n).
\end{equation}
Thus, solving \eqref{rangen+1} is equivalent to find a solution of
\begin{equation}\label{rangelin}
\mathfrak{L}_{n+1}(\varepsilon,w_n)[h]=\varepsilon (r_{n}+\mathcal{R}_n(h)),\quad h \in W^{(n+1)}.
\end{equation}
\begin{lemma}\label{particolare}
There exists $\varepsilon_0:=\varepsilon_0(\gamma,\tau,\bar{\sigma},s)>0$ small enough such that for any $ \varepsilon \in \mathcal{G}_{n+1}$ defined in \eqref{Adritto}, 
the operator $\mathfrak{L}_{n+1}(\varepsilon,w_n)$ is invertible on $W^{(n+1)}$ and \eqref{stimalinearizzato} holds at the $n+1$ step.
\end{lemma}

\begin{proof}
It is sufficient to prove
\begin{equation}\label{normaalta}
\|w_n\|_{\sigma_{n+1},s+\beta}\leq \frac{\varepsilon}{\gamma}K(\tau),\quad \text{where }\beta:=\frac{2\tau(\tau-1)}{2-\tau}.
\end{equation}
Indeed, by \eqref{normaalta}, taking $\varepsilon_0\gamma^{-1}K(\tau)< \rho$, the Lemma follows applying Proposition \ref{inversolinearizzato}.\\
We introduce the following loss in analiticity-gain in Sobolev estimate
\begin{equation}\label{stimaesp}
\|u\|_{\sigma-\alpha,s+\beta}\leq  \max \left\lbrace 1, \exp(-\beta)\left(\frac{\beta}{\alpha}\right)^\beta \right\rbrace\|u\|_{\sigma, s},\quad \forall \alpha,\beta \geq 0
\end{equation}
which follows from
$$
\sup\limits_{x\geq 0} \exp(-\alpha x) \langle x \rangle^\beta\leq \max \left\lbrace 1, \exp(-\beta)\left(\frac{\beta}{\alpha}\right)^\beta \right\rbrace,\,\, \forall \alpha,\beta \geq 0.
$$
We expand $w_n=\sum\limits_{i=0}^n h_i$, and using \eqref{stimaesp} on each $h_i$ with $\sigma=\sigma_i$, $\alpha=\sigma_i-\sigma_{n+1},$ $\beta=\frac{2\tau(\tau-1)}{2-\tau}$ we obtain
$$
\|h_i\|_{\sigma_{n+1},s+\beta}\leq \max \left\lbrace 1, \exp(-\beta)\left(\frac{\beta}{\sigma_i-\sigma_{n+1}}\right)^\beta \right\rbrace \|h_i\|_{\sigma_i,s} \, , 
$$
and, using \eqref{stimacca}, 
$$
\|w_n\|_{\sigma_{n+1},s}\leq \sum\limits_{i=0}^{n} \frac{\varepsilon}{\gamma}\exp(-\chi^i)\max \left\lbrace 1, \exp(-\beta)\left(\frac{\beta(1+i^2)}{\theta}\right)^\beta \right\rbrace\leq K(\tau)\frac{\varepsilon}{\gamma}.
$$
The lemma is proved.
\end{proof}

We can now solve Equation \eqref{rangelin} by contraction.

\begin{lemma}[Contraction]\label{lemmacontraction}
There exist $L_0:=L_0(\gamma,\tau,\bar{\sigma},s)>0,\, \varepsilon_0 := \varepsilon_0(\gamma,\tau,\bar{\sigma},s)>0$ such that for any $ \varepsilon \in \mathcal{G}_{n+1}$ defined in \eqref{Adritto}, there exists a unique solution $ h_{n+1}(\varepsilon)\in W^{(n+1)}$ of $\eqref{rangelin}$ satisfying $\|h_{n+1}\|_{\sigma_{n+1},s}\leq \frac{\varepsilon}{\gamma}\exp\left(-\chi^{n+1}\right)$. It follows that $w_{n+1}:=w_n+h_{n+1}$ is solution of $\eqref{wenne}$ at the $n+1$ step.
\end{lemma}
\begin{proof}
Remembering the definitions \eqref{resto1}, \eqref{resto2} we are going to prove that the map
$$
\mathcal{T}_{\varepsilon}^{(n+1)}(h):=\varepsilon\mathfrak{L}_{n+1}(\varepsilon, w_{n})^{-1}(r_{n}+\mathcal{R}_n(h))
$$
is a contraction on the set
$$
 \mathcal{D}_{n+1}:=\left\lbrace h\in W^{(n+1)}\, :\, \|h\|_{\sigma_{n+1},s}\leq \rho_{n+1} \right\rbrace, \quad \rho_{n+1}:=\frac{\varepsilon}{\gamma}\exp(-\chi^{n+1}).
$$
{\it Step 1: $\mathcal{T}_{\varepsilon}^{(n+1)}(\varepsilon,\cdot)$ maps the disk $\mathcal{D}_{n+1}$ into itself.}\\
By  \eqref{R'} and the smoothing estimates \eqref{smooth}, we have
\begin{equation}\label{rpiccolo}
\|r_{n}\|_{\sigma_{n+1},s}\leq \exp(-L_n(\sigma_n-\sigma_{n+1}))\|P_{n+1}\Pi_W\Gamma(w_n)\|_{\sigma_n,s}\leq \exp(-L_n \gamma_n) R'.
\end{equation}
In order give a bound of the term $\mathcal{R}_n(h)$ we write it in integral Taylor remainder form:
\begin{equation}\label{integral}
\mathcal{R}_n(h)=P_{n+1}\Pi_W Z(h)[h,h],\quad \quad Z(h):= \int\limits_0^1 t\int\limits_0^1 D^2\Gamma(w_n+sth)dsdt.
\end{equation}
Then one has $\|\mathcal{R}_n(h)\|_{\sigma_n,s}\leq \sup\limits_{t\in[0,1]}\|D^2\Gamma(w_n+th)\|_{\sigma_{n+1},s}\|h\|^2_{\sigma_{n+1},s}$ and  by \eqref{R'} we obtain
\begin{equation}\label{rgrande}
\| \mathcal{R}_n(h)\|_{\sigma_{n+1},s}\leq R'\|h\|^2_{\sigma_{n+1},s} \, .
\end{equation}
For any $\varepsilon \in  \mathcal{G}_{n+1}$ it follows by Lemma \ref{particolare}, \eqref{rpiccolo}, \eqref{rgrande} for all $h\in \mathcal{D}_{n+1}$
\begin{align*}
\left\|\mathcal{T}_{\varepsilon}^{(n+1)}(h)\right\|_{\sigma_{n+1},s}\leq \frac{K\varepsilon}{\gamma}(L_{n+1})^{\tau-1}\left( \|r_{n}\|_{\sigma_{n+1},s}+\|\mathcal{R}_n(h)\|_{\sigma_{n+1},s}\right)\\
\leq \varepsilon\frac{C'}{\gamma}L_{n+1}^{\tau-1}(\exp(-L_n \gamma_n)+\rho_{n+1}^2).
\end{align*}
Since $\rho_{n+1}=\frac{\varepsilon}{\gamma}\exp(-\chi^{n+1})$ one has:
\begin{itemize}
\item{$\varepsilon\frac{C'}{\gamma}L_{n+1}^{\tau-1}\exp(-L_n \gamma_n)=\varepsilon\frac{C'}{\gamma}L_{0}^{\tau-1}2^{(n+1)(\tau-1)}\exp\left(-L_0\theta\frac{2^{n+1}}{1+n^2}\right)\leq \frac{\rho_{n+1}}{2}$, in fact this holds if and only if $L_0^{\tau-1}\exp\left(-L_0\theta\frac{2^{n+ 1}}{1+n^2}\right)\leq C'^{-1}2^{-(n+1)(\tau-1)}\exp(-\chi^{n+1}), \,\, \forall n \in \N$.\\
The left hand side is way smaller for large $n$ since $\frac{2^n}{1+n^2}\gg \chi^{n+1}$, and so if $L_0$ is chosen large enough the inequality holds $\forall n\in \N$.
}
\item{$C'\frac{\varepsilon}{\gamma}(2^{n+1}L_0)^{\tau-1}\rho_{n+1}^2\leq \frac{\rho_{n+1}}{2}$, in fact by simplifying one has the equivalent condition\\
$$\frac{\varepsilon^2}{\gamma^2}\leq \inf\limits_{n\in \N} \frac{\exp(\chi^{n+1})}{2^{1+(n+1)(\tau-1)}C'L_0^{\tau-1}}$$
which is clearly satisfied if $\varepsilon_0\gamma^{-1}$ is small enough since the right hand side goes to $\infty$ as $n\rightarrow\infty$.\\
}
\end{itemize}
{\it Step 2: $\mathcal{T}_{\varepsilon}^{(n+1)}$ is a contraction on the disk $\mathcal{D}_{n+1}$.}
By \eqref{integral} we have
$$\mathcal{R}_n(h)-\mathcal{R}_n(h')=Z(h)[h,h]-Z(h')[h',h']=V(h)[h,h-h']+Z(h)[h-h',h']+(Z(h)-Z(h'))[h',h'].$$
Now the terms $Z(h),Z(h')$ are bounded by $\|D^2\Gamma\|$, while $Z(h)-Z(h')$ is bounded by $\|D^3\Gamma\|\|h-h'\|$, thus by \eqref{R'} we have for all $h\in \mathcal{D}_{n+1}$
\begin{equation}\label{rgrande2}
\|\mathcal{R}_n(h)-\mathcal{R}_n(h')\|_{\sigma_{n+1},s}\leq 2R'(\|h\|_{\sigma_{n+1},s}+\|h'\|_{\sigma_{n+1},s})\|h-h'\|_{\sigma_{n+1},s}.
\end{equation}
Hence, we obtain by \eqref{stimalinearizzato}, \eqref{rgrande2}:
\begin{align*}
&\left\|\mathcal{T}_{\varepsilon}^{(n+1)}(h)-\mathcal{T}_{\varepsilon}^{(n+1)}(h')\right\|_{\sigma_{n+1},s}=\left\|\varepsilon\mathfrak{L}_{n+1}(\varepsilon,w_n)(\mathcal{R}_n(h)-\mathcal{R}_n(h'))\right\|_{\sigma_{n+1},s}\\
&\leq \varepsilon \frac{C}{\gamma}(L_{n+1})^{\tau-1}C'(\|h\|_{\sigma_{n+1},s}+\|h'\|_{\sigma_{n+1},s})\|h-h'\|_{\sigma_{n+1},s}\\
&\leq C''\frac{\varepsilon^2}{\gamma^2} (L_{n+1})^{\tau-1}\exp(-\chi^{n+1})\|h-h'\|_{\sigma_{n+1},s}.
\end{align*}
It follows that for $\frac{\varepsilon_0^2}{\gamma^2}\leq \inf\limits_{n\in \N}\frac{\exp(\chi^{n+1})}{C''2^{1+(n+1)(\tau-1)}L_0^{\tau-1}}$ one has that $\mathcal{T}_{n+1}(\varepsilon, \cdot)$ is a contraction.
\end{proof}
In the next Lemma we show that for $L_0$ big enough and $\varepsilon_0$ small enough, both depending on $\gamma$, estimates \eqref{stimacca} hold for a constant $\bar{C}$ not depending on $\gamma$.
\begin{lemma}[Estimate of $\partial_{\varepsilon}h_{n+1}(\varepsilon)$]
There exists $L_0 :=L_0(\gamma,\tau,\sigma,s)>0$, $\varepsilon_0 :=\varepsilon_0(\gamma,\tau,\sigma,s) > 0 $ such that the function $\partial_\varepsilon h_{n+1}(\varepsilon)$ in Lemma \ref{lemmacontraction} satisfies \eqref{stimacca} at the $n+1$ step. 
As a consequence the function $w_{n+1}=\sum\limits_{i=0}^{n+1}h_{i}(\varepsilon)$ satisfies \eqref{estiwn} at $n+1$ step.
\end{lemma}

\begin{proof}
The function $h_{n+1}$ is the unique solution in $\mathcal{D}_{n+1}$ of $U_{n+1}(\varepsilon,h_{n+1})=0$, where
\begin{equation}\label{defU}
U_{n+1}(\varepsilon, h):=\mathcal{L}_\omega (w_n + h)-\varepsilon P_{n+1}\Pi_W \Gamma(w_n+h).
\end{equation}
By differentiating \eqref{defU} we have $D_h U_{n+1}(\varepsilon, h_{n+1})=\mathcal{L}_\omega-\varepsilon P_{n+1}\Pi_W D_w\Gamma(w_n+h_{n+1})=\mathfrak{L}_{n+1}(\varepsilon,w_{n+1})$.\\
By Lemma \ref{particolare} if $\varepsilon \in \mathcal{G}_{n+1}$, then $\mathfrak{L}_{n+1}(\varepsilon,w_n)$ is invertible and \eqref{stimalinearizzato} holds.\\
Moreover $\mathfrak{L}_{n+1}(\varepsilon,w_{n+1})-\mathfrak{L}_{n+1}(\varepsilon,w_{n})=\varepsilon P_{n+1}\Pi_W(D_w\Gamma(w_{n+1})-D_w\Gamma(w_n))$, hence by \eqref{R'}:
\begin{equation}\label{dis2t}
\left\| \mathfrak{L}_{n+1}(\varepsilon,w_{n+1})h_{n+1}-\mathfrak{L}_{n+1}(\varepsilon,w_{n})h_{n+1}   \right\|_{\sigma_{n+1},s}\leq  \varepsilon R' \|h_{n+1}\|_{\sigma_{n+1},s}\leq R'\varepsilon^2\gamma^{-1}\exp(-\chi^{n+1}).
\end{equation}
We rewrite $D_hU_{n+1}(\varepsilon,h_{n+1})=	\mathfrak{L}_{n+1}(\varepsilon,w_{n+1})$ in the following way:
\begin{align*}
&\mathfrak{L}_{n+1}(\varepsilon,w_{n+1})=\mathfrak{L}_{n+1}(\varepsilon,w_{n})+\mathfrak{L}_{n+1}(\varepsilon,w_{n+1})-\mathfrak{L}_{n+1}(\varepsilon,w_{n})\\
&=\mathfrak{L}_{n+1}(\varepsilon,w_n)\left( \mathds{1} +\mathfrak{L}_{n+1}(\varepsilon,w_{n})^{-1}\left(\mathfrak{L}_{n+1}(\varepsilon,w_{n+1})-\mathfrak{L}_{n+1}(\varepsilon,w_{n})\right)\right).
\end{align*}
Using \eqref{stimalinearizzato}, \eqref{dis2t}, we obtain:
$$\left\|\mathfrak{L}_{n+1}(\varepsilon,w_{n})^{-1}\left(\mathfrak{L}_{n+1}(\varepsilon,w_{n+1})-\mathfrak{L}_{n+1}(\varepsilon,w_{n})\right)\right\|_{B\left(X_{\sigma_{n+1},s}\right)}\leq KR'\varepsilon^2 \gamma^{-2}(L_{n+1})^{\tau-1}\exp(-\chi^{n+1}).$$
Thus, under the smallness assumption $\varepsilon_0 \gamma^{-1} \leq \left[(KR')^{-1}L_0^{-\tau+1} \inf\limits_{n\in \N}2^{-1-(n+1)(\tau-1)}\exp(\chi^{n+1})\right]^{\frac{1}{2}}$, we can invert $(\mathds{1} +\mathfrak{L}_{n+1}(\varepsilon,w_{n})^{-1}\left(\mathfrak{L}_{n+1}(\varepsilon,w_{n+1})-\mathfrak{L}_{n+1}(\varepsilon,w_{n}))\right)$ by Neumann Series, obtaining that $D_h U_{n+1}(\varepsilon,h_{n+1})$ is invertible, and again by \eqref{stimalinearizzato} its inverse satisfies
\begin{equation}\label{diffU}
\left\|D_h U_{n+1}(\varepsilon,h_{n+1})^{-1}h\right\|_{\sigma_{n+1},s}\leq \frac{2K}{\gamma}(L_{n+1})^{\tau-1}\|h\|_{\sigma_{n+1},s},\,\, \forall h\in W^{(n+1)},\, \forall \varepsilon \in \mathcal{G}_{n+1}.
\end{equation}
This proves that $h_{n+1}\in C^1(\mathcal{G}_{n+1},W^{(n+1)})$, since it is the implicit function defined by $U_{n+1}(\varepsilon,h_{n+1})=0$ which has invertible differential in $h$ by \eqref{diffU}. 
Now we estimate $\partial_\varepsilon U_{n+1}(\varepsilon,h)$, in order to do that we write $U_{n+1}(\varepsilon,h)$ in the following way (using $\mathcal{L}_\omega w_n=\varepsilon P_n\Pi_W\Gamma(w_n)$):
$$
U_{n+1}(\varepsilon,h)=\mathcal{L}_{\omega}h-\varepsilon P_{n+1}\Pi_W \left(\Gamma(w_n+h)-\Gamma(w_n)\right)+\varepsilon P_n^{\perp}P_{n+1}\Pi_W \Gamma(w_n),
$$
and by computing its derivatives (remember that $\omega^2=1+\varepsilon)$ we obtain
\begin{align*}
&\partial_\varepsilon U_{n+1}(\varepsilon,h)=-\partial_{tt}h-P_{n+1}\Pi_W\left(\Gamma(w_n+h)-\Gamma(w_n)\right)+ P_n^{\perp}P_{n+1}\Pi_W \Gamma(w_n)\\
&+\varepsilon \left( -P_{n+1}\Pi_W\left(\partial_w\Gamma(w_n+h)-\partial_w\Gamma(w_n)\right)\partial_\varepsilon w_n + P_n^{\perp}P_{n+1}\Pi_W \partial_w\Gamma(w_n)\partial_\varepsilon w_{n} \right).
\end{align*}
Using \eqref{R'}, \eqref{smooth} one estimates each of these terms
\begin{equation}\label{stimeDU}
\|\partial_\varepsilon U_{n+1}(\varepsilon,h_{n+1})\|_{\sigma_{n+1},s}\leq R'\left(1+\varepsilon\|\partial_\varepsilon w_n(\varepsilon))\|_{\sigma_n,s}\right)\left(\exp(-L_n\gamma_n)+\rho_{n+1}\right)+L_{n+1}^2\rho_{n+1}.
\end{equation}
It follows by the Implicit function theorem and estimates \eqref{diffU}, \eqref{stimeDU} that
\begin{equation}\label{acaso}
\|\partial_\varepsilon h_{n+1}\|_{\sigma_{n+1},s}\leq \frac{2K}{\gamma}L_{n+1}^{\tau+1}\rho_{n+1}+\frac{2KR'L_{n+1}^{\tau-1}}{\gamma}\left(1+\varepsilon\|\partial_\varepsilon w_n(\varepsilon))\|_{\sigma_n,s}\right)\left(\exp(-L_n\gamma_n)+\rho_{n+1}\right).
\end{equation}
It follows that for $L_0(\gamma, \tau, \bar{\sigma},s)$ big enough and $\varepsilon_0(\gamma,\tau, \bar{\sigma},s)$ small enough, given $\bar{\chi}\in \left]0,\chi\right[$, $\exists \bar{C}=\bar{C}(\tau,\bar{\sigma},s,\bar{\chi})$ large enough such that $\|\partial_\varepsilon h_{n+1}\|_{\sigma_{n+1},s}\leq \bar{C}\exp(-\bar{\chi}^{n+1}),\forall n\in \N$. Indeed by \eqref{acaso} it is sufficient to choose $\bar{C}$ which satisfies the following inequalities:
\begin{itemize}
\item{$\frac{2K}{\gamma}L_{n+1}^{\tau+1}\rho_{n+1}\leq \frac{1}{3}\bar{C} \exp(-\bar{\chi}^{n+1})$.\\
We take $\bar{C}\geq \varepsilon\gamma^{-2}6K'\max\limits_{n\in\N}L_{n+1}^{\tau+1}\exp\left(-\chi^{n+1}+\bar{\chi}^{n+1}\right)$, which can be done since the right hand side goes to 0 as $n\to \infty$.
}
\item{$\frac{4KR'}{\gamma}L_{n+1}^{\tau-1}\left(1+3\varepsilon \bar{C}\right)\rho_{n+1}\leq \frac{1}{3}\bar{C} \exp(-\bar{\chi}^{n+1}).$\\
We take $\frac{1}{3\varepsilon_0}\geq \bar{C}\geq \frac{24KR'}{\gamma}\varepsilon\gamma^{-1}\max\limits_{n\in\N}L_{n+1}^{\tau-1}\exp(-\chi^{n+1}+\bar{\chi}^{n+1})$.\\
The r.h.s goes to $0$ as $n\to \infty$, hence there exists a constant $\bar{C}$ which is upper bound $\forall n\in \N$. Then for $\varepsilon_0$ small enough both inequalities are satisfied and compatible.
}
\item{$\frac{4KR'}{\gamma}L_{n+1}^{\tau-1}\left(1+3\varepsilon \bar{C}\right)\exp(-L_n\gamma_n)\leq \frac{1}{3}\bar{C} \exp(-\bar{\chi}^{n+1}).$\\
As we did before, we take $\frac{1}{3\varepsilon_0}\geq \bar{C}\geq  \frac{24KR'}{\gamma}\max\limits_{n\in \N}L_{n+1}^{\tau-1} \exp\left(-\frac{\theta L_02^n}{1+n^2}+\bar{\chi}^{n+1}\right)$.
Since the $\max$ exists finite, we can pick such a $\bar{C}>0$.
}
\end{itemize}

By taking a constant $\bar{C}$ for which the previous conditions hold, we obtain that\\
$\|\partial_\varepsilon h_{n+1}(\varepsilon)\|_{\sigma_{n+1},s}	\leq \bar{C} \exp(-\bar{\chi}^{n+1}),\,\, \forall n \in \N$, and so \eqref{stimacca} is proved.\\
As a consequence $\|w_{n+1}(\varepsilon)\|_{\sigma_{n+1},s}\leq \sum\limits_{i=0}^{n+1}\frac{\varepsilon}{\gamma}\exp(-\chi^{i})\leq K_1 \frac{\varepsilon}{\gamma},$ with $K_1=\sum\limits_{i=0}^{\infty}\exp(-\chi^i) $, and\\ $\|\partial_\varepsilon w_{n+1}(\varepsilon)\|_{\sigma_{n+1},s}\leq \sum\limits_{i=0}^{n+1}\|\partial_\varepsilon h_i(\varepsilon)\|_{\sigma_i,s}\leq \bar{C}\sum\limits_{i=0}^{\infty}\exp(-\bar{\chi}^i)=: K_1'$.
\end{proof}
We now extend the functions $w_n(\varepsilon)$, which are defined for $\varepsilon \in \mathcal{G}_n$, to the whole set $\varepsilon \in [0,\varepsilon_0]$.

\begin{lemma}[Whitney $C^1$-Extension]\label{wit}
For any $i\in \N$ 
there exists $ \,\tilde{h}_i \in C^1([0,\varepsilon_0],W^{(i)} \cap X_{\sigma_i,s})$ satisfying $\|\tilde{h}_i\|_{\sigma_i,s}\leq \frac{K\varepsilon}{\gamma}\exp(-\tilde{\chi}^i),$ for some $ \tilde{\chi} \in ]1,\bar{\chi}[,$ and such that $\tilde{w}_n:=\sum\limits_{i=0}^n\tilde{h}_i \in C^1([0,\varepsilon_0],W^{(n)}\cap X_{\sigma_n,s})$ satisfies items \ref{tilde1}, \ref{tilde2} of Proposition \ref{mainpropw}.
\end{lemma}
\begin{proof}
Let $\phi:\R \rightarrow \R$ be a smooth cutoff such that $0\leq \phi\leq 1$, $\operatorname{supp} \phi \subseteq ]-1,1[$, $\int\limits_{\R} \phi(\varepsilon) d\varepsilon=1$.\\
Let us take $\nu>0$ and define the rescaled function $\phi_i(\varepsilon):=\frac{L_i^3}{\nu} \phi\left(\frac{L_i^3}{\nu}\varepsilon \right)$, so that $\operatorname{supp} \phi_i \subseteq \left]-\frac{\nu}{L_i^3},\frac{\nu}{L_i^3}\right[$ and $\int\limits_{\R}\phi_i(\varepsilon) d\varepsilon=1.$\\
Now let $\psi_i(\varepsilon):=\phi_i * \chi_{\tilde{\mathcal{G}_i}}(\varepsilon)=\int\limits_\R \phi_i(\varepsilon-\eta)\chi_{\tilde{\mathcal{G}}_i}(\eta)d\eta$.\\
One has $0\leq \psi_i \leq 1$, $\operatorname{supp} \psi_i \subseteq \operatorname{supp} \phi_i + \tilde{\mathcal{G}}_i\subset\subset \operatorname{int}(\mathcal{G}_i)$, moreover $\psi_i \in C^{\infty}$, and\\
$\left|\partial_\varepsilon\psi_i\right|\leq C \frac{L_i^3}{\nu}$, where $C=\left\|\partial_\varepsilon\phi\right\|_{L^1}$ since
\begin{equation}\label{derivata}
\partial_\varepsilon\psi_i(\varepsilon)= \left(\partial_\varepsilon\phi_i\right)*\chi_{\tilde{\mathcal{G}}_i}(\varepsilon)=\left( \frac{L_i^3}{\nu}\right)^2\int\limits_\R\left(\partial_\varepsilon\phi\right)\left( \frac{L_i^3}{\nu}(\varepsilon-\eta)\right)\chi_{\tilde{\mathcal{G}}_i}(\eta)d\eta.
\end{equation}
Now we can define 
$$\tilde{w}_0(\varepsilon):=w_0(\varepsilon),\quad \quad \tilde{w}_{i+1}:=\tilde{w}_i+\tilde{h}_{i+1}\in W^{(i+1)},\quad \text{with}\quad\tilde{h}_{i+1}:=\chi_{\mathcal{G}_{i+1}}(\varepsilon)\psi_{i+1}(\varepsilon)h_{i+1}.$$
We have $\tilde{h}_{i+1}\in C^1\left([0,\varepsilon_0],W^{(i+1)}\right)$ since $h_i$ is differentiable on $\mathcal{G}_i$, and $\operatorname{supp} \psi_i \subset \subset \operatorname{int}(\mathcal{G}_i)$, furthermore $\|\tilde{h}_i(\varepsilon)\|_{\sigma_i,s}\leq \|\psi_i\|_{L^\infty[0,\varepsilon_0]} \|h_i(\varepsilon)\|_{\sigma_i,s}\leq \bar{C}\varepsilon\gamma^{-1}\exp(-\tilde{\chi}^i)$ by \eqref{stimacca}. It follows that:
$$\|\tilde{w}_n\|_{\sigma_n,s}\leq \sum\limits_{i=0}^{n}\|\tilde{h}_i\|_{\sigma_i,s}\leq K_2 \gamma^{-1}\varepsilon.$$
Then, by chain rule, by \eqref{derivata} and \eqref{stimacca}, $\|\partial_\varepsilon\tilde{h}_i\|_{\sigma_i,s}\leq \tilde{K}_2\nu^{-1}\exp(-\tilde{\chi}^i),$ for some $\tilde{\chi}\in ]1,\bar{\chi}[$. It follows $\|\partial_\varepsilon \tilde{w}_n(\varepsilon)\|_{\sigma_n,s}\leq K_2\nu^{-1},\, \forall n\in \N,\, \forall \varepsilon \in [0,\varepsilon_0]$.
\end{proof}
With this we completed the proof of each point of Proposition \ref{mainpropw}. Now we prove Theorem \ref{mainthmw}.
\begin{lemma}[Existence and estimates of $\tilde{w}(\varepsilon)$]\label{existencedelta}
The function $\tilde{w}(\varepsilon):=\sum\limits_{i\in \N} \tilde{h_i}(\varepsilon)=\lim\limits_{n\to \infty}\tilde{w}_n(\varepsilon)$, where $\tilde{h}_i(\varepsilon)$ were defined in Lemma \ref{wit}, satisfies the first estimate of \eqref{stimetilda}. Furthermore
\begin{equation}\label{ratio}
\|\tilde{w}(\varepsilon)-\tilde{w}_n(\varepsilon)\|_{\sigma_\infty,s}\leq \hat{K}\frac{\varepsilon}{\gamma}\exp(-\tilde{\chi}^n).
\end{equation}
\end{lemma} \begin{proof}
By Lemma \ref{wit} we have $\|\tilde{w}(\varepsilon)\|_{\sigma_\infty,s}\leq \sum\limits_{i=0}^\infty \|\tilde{h}_i(\varepsilon)\|_{\sigma_i,s}\leq K_2\frac{\varepsilon}{\gamma}$ so $\tilde{w}(\varepsilon)\in X_{\frac{\bar{\sigma}}{2},s}$, and we also have
$\|\partial_\varepsilon \tilde{w}(\varepsilon)\|_{\sigma_{\infty},s}\leq  \sum\limits_{i=0}^\infty  \|\partial_\varepsilon \tilde{h}_i(\varepsilon)\|_{\sigma_i,s}\leq K_2\nu^{-1}$, which proves \eqref{stimetilda} since $\sigma_\infty>\frac{\bar{\sigma}}{2}$, and similarly \eqref{ratio} follows.
\end{proof}
Now, in order to complete the proof of Theorem \ref{mainthmw} we introduce the following sets:
\begin{equation}\label{cantorBn}
\begin{gathered}
B_n:=\bigg\lbrace \varepsilon\in [0,\varepsilon_0]\, :\, \left| \omega(\varepsilon)\ell -\omega_j-\varepsilon \frac{M(\tilde{w}(\varepsilon))}{2\omega_j}\right|\geq \frac{2\gamma}{(\ell+\omega_j)^\tau},\\
|\omega(\varepsilon)\ell-\omega_j|\geq \frac{2\gamma}{(\ell+\omega_j)^\tau};\,\, \forall \ell,j\in \N\,:\, \frac{1}{3\varepsilon}\leq\ell\leq L_n,\,\omega_j\leq 2L_n,\, \ell\neq \omega_j\,  \bigg\rbrace .
\end{gathered}
\end{equation}
We remark that the set $B_\infty$ defined in \eqref{cantor} is the limit set of these sets $B_n$, namely $B_\infty:=\bigcap\limits_{n\in\N}B_n$.\\
We also remark that these sets depend only on the limit function $\tilde{w}(\varepsilon)$, while the sets $\mathcal{G}_n$, $\tilde{\mathcal{G}}_n$, respectively defined in \eqref{Adritto},\eqref{Atildati} depend on the $n$-th iterated $w_n(\varepsilon)$. For this reason the sets $B_n$ are much easier to handle if we want to give a measure estimate. Next we show that $\forall n\in \N,\, B_n\subseteq \tilde{\mathcal{G}}_n$, and that $B_\infty$ has actually positive measure.

\begin{lemma}\label{gammanu}
There exist $\nu_0 :=\nu_0(\tau,\bar{\sigma}, s)>0$ and $\varepsilon_0 :=\varepsilon_0(\gamma,\tau,\bar{\sigma},s)>0$ such that for $0<\varepsilon < \varepsilon_0$ and $ 0<\nu\gamma^{-1}<\nu_0$ then  $B_n \subset \tilde{\mathcal{G}}_n,\, \forall n \in \N.$
\end{lemma}

\begin{proof}
First of all, by definition  $B_0 \subset \tilde{\mathcal{G}}_0=[0,\varepsilon_0]$, now assume as induction hypothesis that $B_n \subset \tilde{\mathcal{G}}_n$, our aim is to prove that $B_{n+1} \subset \tilde{\mathcal{G}}_{n+1}$. 
In order to prove that, it is sufficient to show that $\forall \varepsilon\in B_{n+1},$ the ball \\$D_{\nu, n+1}(\varepsilon):=\left\lbrace \varepsilon'\in [0,\varepsilon_0]\, :\, |\varepsilon-\varepsilon'| < \frac{2\nu}{L_{n+1}^3} \right\rbrace$ satisfies the inclusion $D_{\nu, n+1}(\varepsilon) \subset \mathcal{G}_{n+1}$.\\
We have by induction assumption that $B_{n+1} \subseteq B_n \subset \tilde{\mathcal{G}}_n$, and since $L_{n+1}>L_n$ then  the ball $D_{\nu,n+1}(\varepsilon)\subset D_{\nu,n}(\varepsilon) \subset \mathcal{G}_{n}$, since $\varepsilon 	\in \tilde{\mathcal{G}}_n$.\\
Now consider $\varepsilon'\in D_{\nu, n+1}(\varepsilon)$, since $\varepsilon\in \tilde{\mathcal{G}}_n$ we have $\tilde{w}_n(\varepsilon)=w_n(\varepsilon)$, then by \eqref{estiwn} and \eqref{ratio} we obtain:

\begin{equation}\label{stimewu}
\|w_n(\varepsilon')-\tilde{w}(\varepsilon)\|_{\sigma_\infty,s}\leq \tilde{K}\left(\frac{\nu}{L_n^3}+\frac{\varepsilon}{\gamma}\exp(-\bar{\chi}^{n})\right).
\end{equation}
Then, letting $\omega=\sqrt{1+\varepsilon},\, \omega'=\sqrt{1+\varepsilon'}$ and using the definiton of $B_n$ in \eqref{cantorBn} and the estimates \eqref{stimewu} we obtain:
\begin{align*}
&\left| \omega'\ell-\omega_j-\varepsilon'\frac{M(w_n(\varepsilon'))}{2\omega_j}\right|\geq \left| \omega\ell-\omega_j-\varepsilon'\frac{M(\tilde{w}(\varepsilon))}{2\omega_j}\right|-|(\omega-\omega')\ell|-\frac{\varepsilon'}{2\omega_j}|M(w_n(\varepsilon'))-M(w(\varepsilon))|\\
&\geq\frac{2\gamma}{(\ell+\omega_j)^{\tau}}-\frac{1}{2}|\varepsilon-\varepsilon'||\ell|- \frac{\tilde{C}\varepsilon'}{\omega_j}\left(\|w_n(\varepsilon')-w(\varepsilon)\|_{\sigma_\infty,s}\right)\\
&\geq \frac{2\gamma}{(\ell+\omega_j)^{\tau}}-\bar{C}\left(\frac{\nu}{L_n^3}\left(|\ell|+\frac{\varepsilon'}{\omega_j}\right)+ \frac{\varepsilon'^2}{\gamma \omega_j}\exp(-\bar{\chi}^n)\right)\\
&\geq \frac{2\gamma}{(\ell+\omega_j)^{\tau}}-\bar{C}\gamma\left(\frac{\nu\gamma^{-1}}{L_n^3}\left(|\ell|+\frac{\varepsilon'\gamma^{-1}}{\omega_j}\right)+ \frac{(\varepsilon'\gamma^{-1})^2}{\omega_j}\exp(-\bar{\chi}^n)\right)\\
&\geq \frac{2\gamma}{(\ell+\omega_j)^{\tau}}-\bar{C}\frac{\gamma}{L_n^2}\left(2\nu\gamma^{-1}+(\varepsilon'\gamma^{-1})^2L_{n}^2\exp(-\bar{\chi}^n)\right)\\
&\geq\ \frac{2\gamma}{(\ell+\omega_j)^{\tau}}-9\bar{C}\frac{\gamma}{(\ell+\omega_j)^2}\left(2\nu\gamma^{-1}+(\varepsilon'\gamma^{-1})^2L_{n}^2\exp(-\bar{\chi}^n)\right)\\
&\geq \frac{2\gamma}{(\ell+\omega_j)^\tau}-\frac{\gamma}{(\ell +\omega_j)^2}>\frac{\gamma}{(\ell+\omega_j)^{\tau}}.
\end{align*}
Where we used that the conditions $\ell \leq L_{n+1},\,  \omega_j \leq 2L_{n+1}$ imply $-\frac{1}{L_n^2}\geq -\frac{9}{(\ell+\omega_j)^2}$ for the second last passage, and then we used $ \tau <2$ and the smallness assumptions $\nu\gamma^{-1}\leq (36\bar{C})^{-1},\,\varepsilon_0 \gamma^{-1}\leq \left(\frac{1}{18\bar{C}}\min\limits_{n\in \N}\exp(\bar{\chi}^n)L_n^{-2}\right)^{\frac{1}{2}}$ for the last passage.\\
In the same way one proves the second condition $\left| \omega'\ell-\omega_j\right|> \frac{\gamma}{(\ell+\omega_j)^\tau}$.\\
This completes the proof that $B_{n+1}\subset \tilde{\mathcal{G}}_{n+1}$.
\end{proof}
\begin{proof}[Proof of Theorem \ref{mainthmw}]
Since $B_n\subset \tilde{\mathcal{G}_n},\, \forall n \in \N$, then if $\varepsilon\in \bigcap\limits_{n=0}^{\infty}B_n$ we have that $\forall n\in \N,\,\tilde{w}_n(\varepsilon)=w_n(\varepsilon)$ is solution of \eqref{rangen+1}, and so the following identity holds:
\begin{equation}\label{nuovaequazione}
w_{n}(\varepsilon)=\varepsilon\mathcal{L}_\omega^{-1}P_n\Pi_W\Gamma(w_n(\varepsilon))=\varepsilon\mathcal{L}_\omega^{-1}\Pi_W\Gamma(w_n(\varepsilon))-\varepsilon\mathcal{L}_\omega^{-1}P_n^{\perp}\Pi_W\Gamma(w_n(\varepsilon)).
\end{equation}
Then one has:
\begin{align*}
&\|\varepsilon\mathcal{L}_\omega^{-1}P_n^{\perp}\Pi_W\Gamma(w_n(\varepsilon))\|_{\frac{\bar{\sigma}}{2},s}\leq C\frac{\varepsilon}{\gamma}(L_{n})^{\tau-1}\|P_n^{\perp}\Pi_W\Gamma(w_n(\varepsilon))\|_{\frac{\bar{\sigma}}{2},s}\\
&\stackrel{\eqref{rangesplitting}}\leq C\frac{\varepsilon}{\gamma}(L_0 2^n)^{\tau-1}\exp\left(\left(\frac{\bar{\sigma}}{2}-\sigma_n\right)L_02^n\right)\|\Gamma(w_n(\varepsilon))\|_{\sigma_n,s}\leq C' \frac{\varepsilon}{\gamma}(L_0 2^n)^{\tau-1}\exp\left(-\frac{\theta L_02^n}{1+n^2}\right)\stackrel{ n\rightarrow \infty}{\longrightarrow} 0.
\end{align*}
The left hand side of \eqref{nuovaequazione} does converge to $\tilde{w}(\varepsilon)$ in $X_{\frac{\bar{\sigma}}{2},s}$ by \eqref{ratio}. Moreover $\mathcal{L}_\omega^{-1}\Pi_W\Gamma(w_n(\varepsilon))$ does converge to $\mathcal{L}_\omega^{-1}\Pi_W\Gamma(\tilde{w}(\varepsilon))$ in $X_{\frac{\bar{\sigma}}{2},s}$ since $\mathcal{L}_{\omega}^{-1}$ is bounded from $X_{\sigma_\infty,s}$ to $X_{\frac{\bar{\sigma}}{2}, s}$. It follows:
$$\varepsilon\mathcal{L}_\omega^{-1}\Pi_W\Gamma(\tilde{w}(\varepsilon))= \lim\limits_{n\rightarrow\infty} \varepsilon\mathcal{L}_\omega^{-1}\Pi_W\Gamma(w_n(\varepsilon))=\lim\limits_{n\rightarrow\infty} w_n(\varepsilon)=\tilde{w}(\varepsilon).$$
Moreover by Lemmas we have that $\tilde{w}$ satisfies the second estimates of \eqref{stimetilda} since by Lemma \ref{existencedelta} $\|\partial_\varepsilon \tilde{w}(\varepsilon)\|_{\frac{\bar{\sigma}}{2},s}\leq K\nu^{-1}$ and by Lemma \ref{gammanu} $\nu$ can be chosen as $\nu=\frac{\nu_0}{2}\gamma$. It follows that $\|\partial_\varepsilon \tilde{w}(\varepsilon)\|_{\frac{\bar{\sigma}}{2},s}\leq \frac{2K}{\nu_0}\gamma^{-1}$, and so for $K_2\geq\frac{2K}{\nu_0}$ we have estimates \eqref{stimetilda} and so we conclude the proof of Theorem \ref{mainthmw}.
\end{proof}
We now prove that the set $B_\infty$ has positive measure, in particular it has asymptotically full measure at 0.
\begin{proposition}\label{measure}
The set $B_{\infty}$ defined in \eqref{cantor} has asymptotically full measure at $0$, namely it satisfies \eqref{fullmeas}.
\end{proposition}
\begin{proof}
Let $\eta \in ]0,\varepsilon_0[$, and let us estimate the measure of the complementary set 
\begin{align*}
B_{\infty}^c\cap ]0,\eta[=\bigg\lbrace \varepsilon \in ]0,\eta[\, :\, \left|\omega(\varepsilon)\ell-\omega_j-\varepsilon\frac{m(\varepsilon)}{2\omega_j}\right|< \frac{2\gamma}{(\ell + \omega_j)^\tau}\,\, \vee \left|\omega(\varepsilon)\ell-\omega_j\right|< \frac{2\gamma}{(\ell + \omega_j)^\tau},\\
\textit{ for some } \ell,\, j \in \N,\, \ell\geq \frac{1}{3\varepsilon},\, \ell \neq \omega_j \bigg\rbrace
\end{align*}
where we denoted $m(\varepsilon):=M(\tilde{w}(\varepsilon))$, which is a $C^1$ function.
We can write then $$B_\infty^c\cap ]0,\eta[\quad\subseteq\bigcup\limits_{(\ell,j)\in \mathcal{I}_R}S_{\ell,j} \cup R_{\ell,j}$$
$$S_{\ell,j}:=\bigg\lbrace \varepsilon \in \left]\frac{1}{3\ell},\eta\right[\, : \,\bigg| \omega(\varepsilon)\ell-\omega_j-\varepsilon\frac{m(\varepsilon)}{2\omega_j} \bigg| < \frac{2\gamma}{(\ell + \omega_j)^\tau} \bigg\rbrace,$$
$$ R_{\ell,j}:=\bigg\lbrace \varepsilon \in \left]\frac{1}{3\ell},\eta\right[\, : \,| \omega(\varepsilon)\ell-\omega_j| < \frac{2\gamma}{(\ell + \omega_j)^\tau} \bigg\rbrace,$$
$$\mathcal{I}_R:=\left\lbrace (\ell,j)\in \N\times \N\, :\, \ell \neq \omega_j,\, \frac{\omega_j}{\ell}\in [1-4\eta,1+4\eta] \right\rbrace.$$
We observe that the condition $ \frac{\omega_j}{\ell}\in [1-4\eta,1+4\eta]$, appearing  on $\mathcal{I}_R$ is not restrictive, namely if it does not hold, then $S_{\ell,j}=R_{\ell,j}=\emptyset$. Indeed, if we assume $\left| \frac{\omega_j}{\ell}-1\right|>4\eta$, then $|\omega_j-\ell|>4\eta \ell$, and it follows that
$$
|\omega(\varepsilon)\ell-\omega_j |\geq 4\eta\ell-|\omega(\varepsilon)-1|\ell\geq 4\eta \ell-\frac{1}{2}\eta\ell=\frac{7}{2}\eta\ell\geq \frac{7}{6}
$$
$$
\bigg| \omega(\varepsilon)\ell-\omega_j-\varepsilon\frac{m(\varepsilon)}{2\omega_j} \bigg| >4\eta \ell -|\omega(\varepsilon)-1|\ell-C\varepsilon\geq \frac{7}{2}\eta\ell- C\eta\geq 1.
$$
Since $\gamma\in \left]0,\frac{1}{6}\right[$, it follows that the inequalities appearing in the definition of $S_{\ell,j}$ and $R_{\ell,j}$ cannot be satisfied, hence they are empty.\\
We now prove $\left| \bigcup\limits_{(\ell,j)\in \mathcal{I}_R} S_{\ell,j} \right|=o(\eta)$.\\
In order to estimate the measure of each $S_{\ell,j}$ we prove that the $C^1$ functions $f_{\ell,j}(\varepsilon):=\omega(\varepsilon)\ell-\omega_j-\varepsilon\frac{m(\varepsilon)}{2\omega_j}$ satisfy
\begin{equation}\label{effelle}
\partial_\varepsilon f_{\ell,j}(\varepsilon)=\frac{\ell}{2\sqrt{1+\varepsilon}}-\frac{m(\varepsilon)-\varepsilon \partial_\varepsilon m(\varepsilon)}{2\omega_j}\geq \frac{\ell}{4}.
\end{equation}
This follows because, since $(\ell,j)\in \mathcal{I}_R$ we have
$$\left|\frac{m(\varepsilon)}{2\omega_j}\right|\leq \left|\frac{m(\varepsilon)}{2\ell}\right|+\left|\frac{(\ell-\omega_j) m(\varepsilon)}{2\ell\omega_j}\right|\leq \frac{3}{2}\eta |m(\varepsilon)|+\frac{2\eta|m(\varepsilon)|}{\omega_j}\leq C\eta,$$
$$\left| \varepsilon \frac{\partial_\varepsilon m(\varepsilon)}{2\omega_j}\right|\leq C\varepsilon \frac{1}{\nu \omega_j}\leq C'\varepsilon \gamma^{-1}\frac{1}{\omega_j}\leq C''\eta.$$
By \eqref{effelle} it follows that $f_{\ell,j}$ admits $C^1$ inverse function with derivative bounded by $\frac{4}{\ell}$, thus
$$|S_{\ell,j}|\leq\int\limits_{\left\lbrace|f_{\ell,j}(\varepsilon)|<\frac{2\gamma}{(\ell+\omega_j)^\tau}\right\rbrace}d\varepsilon=\int\limits_{\left\lbrace |x|\leq \frac{2\gamma}{(\ell+\omega_j)^\tau}\right\rbrace}df_{\ell,j}^{-1}(x)dx\leq C\frac{\gamma}{\ell(\ell+\omega_j)^\tau}\leq C\frac{\gamma}{\ell^{\tau+1}}.$$
Now given $\ell\in \N$, we estimate the number of $j\in \N$ such that $(\ell,j)\in \mathcal{I}_R$.\\
Using the definition we obtain that if $(\ell,j)\in \mathcal{I}_R$, then $j \in ]\ell-4\eta\ell-1,\ell+4\eta\ell-1[$, hence $\#\left\lbrace j\in \N\, :\,(\ell,j)\in \mathcal{I}_R\right\rbrace \leq 8 \eta\ell$. We can estimate now:
$$\left| \bigcup\limits_{(\ell,j)\in \mathcal{I}_R} S_{\ell,j} \right|\leq \sum\limits_{(\ell,j)\in \mathcal{I}_R}|S_{\ell,j}|\leq  C\sum\limits_{(\ell,j)\in \mathcal{I}_R}\frac{\gamma}{\ell^{\tau+1}}\leq C'\sum\limits_{\ell>\frac{1}{3\eta}}\frac{\gamma \eta \ell}{\ell^{\tau+1}}\leq C''\eta\gamma \sum\limits_{\ell>\frac{1}{3\eta}}\frac{1}{\ell^{\frac{\tau+1}{2}}}\eta^{\frac{\tau-1}{2}}\leq c \gamma \eta^{\frac{\tau+1}{2}}.$$
Hence we proved that it is $o(\eta)$, since $\tau>1$.\\
With similar arguments we prove $\left| \bigcup\limits_{(\ell,j)\in \mathcal{I}_R} R_{\ell,j} \right|=o(\eta)$, using $f_{\ell,j}(\varepsilon)=\omega(\varepsilon)\ell-\omega_j$. This concludes the proof since we have:
$$\lim\limits_{\eta\to 0^+}\frac{\left| B_\infty\cap ]0,\eta[\,\right|}{\eta}\geq\lim\limits_{\eta\to 0^+}\frac{\eta-\left| \bigcup\limits_{(\ell,j)\in \mathcal{I}_R} S_{\ell,j} \right|-\left| \bigcup\limits_{(\ell,j)\in \mathcal{I}_R} R_{\ell,j} \right|}{\eta}= \lim\limits_{\eta\to 0^+}\frac{\eta-o(\eta)}{\eta}=1.$$
\end{proof}
\remark\label{rem315} Since the amplitude to frequency map $\omega(\varepsilon)=\sqrt{1+\varepsilon}$ is smooth near 0 with bounded derivative from below and above by some positive constants, then the set $\omega(B_{\infty})$ has also asymptotically full measure at 1.\\

We finally prove that the solutions of Theorem \ref{teoremone} are $C^\infty$ also in space.
\begin{lemma}\label{classical}
Let $\sigma>0$, $s>\frac{1}{2}$, $r>\frac{3}{2}$ and assume that $u\in X_{\sigma, s,r}$ is a solution of
\begin{equation}\label{eqqq}
    -\omega^2\partial_{tt}u-Au=u^3.
\end{equation}
Then, for any $0<\tilde{\sigma}<\sigma$, it results that $u\in X_{\tilde{\sigma}, s, r+2}$, in particular $u
$ is $C^{\infty}$ also in the  variable $x$.
\end{lemma}

\begin{proof}
Since $u$ is solution of \eqref{eqqq} then $Au=-\omega(\varepsilon)^2\partial_{tt}u-\varepsilon u^3$.
 The algebra estimates \eqref{algebra} and the boundedness of $-\partial_{tt}:X_{\tilde{\sigma},r+2,s}\longrightarrow X_{\tilde{\sigma},r,s}$ imply that
\begin{equation}\label{1diseq}
\|u\|_{\tilde{\sigma},s,r+2}\leq 2\|u\|_{\tilde{\sigma},s+2,r}+\varepsilon C_{s,r}\|u\|^3_{\tilde{\sigma},s,r}.
\end{equation}
Now the embedding $X_{\sigma, s,r}\hookrightarrow X_{\tilde{\sigma},s+2,r}$ is bounded, namely:
\begin{equation}\label{2diseq}
\|u\|_{\tilde{\sigma},s+2,r}\leq C\|u\|_{\sigma,s,r},\quad \bigg(\, C=\max\left\lbrace 1, \sup\limits_{\rho\geq 0}\exp(-\rho(\sigma-\tilde{\sigma})\langle \rho \rangle^{2}	\right\rbrace \bigg).
\end{equation}
Then, by applying \eqref{1diseq} and \eqref{2diseq} one has
\begin{equation}\label{u+reg}
\|u\|_{\tilde{\sigma},s,r+2}\leq 2C\|u\|_{\sigma,s,r}+\varepsilon C_{s,r}\|u\|_{\sigma,s,r}^3<\infty.
\end{equation}
Now one can iterate $m\in \N$ times this procedure by choosing any sequence $\sigma=\sigma_0>\sigma_1>\dots>\sigma_{m}>0$, obtaining $u \in X_{\sigma_m,s,r+2m}$. Since $m$ is arbitrary (and so is the decreasing sequence) one has that $u\in X_{\tilde{\sigma},s,\infty},\, \forall 0<\tilde{\sigma}<\sigma$, namely it is $C^\infty$ also in the space variable up to an arbitrary small shrinking of the strip of analiticity in time, and so the thesis follows.
\end{proof} 
\section{Invertibility of Linearized Operator}\label{invlin}
The aim of this Section is to prove Proposition \ref{inversolinearizzato}. Recalling \eqref{Gammone} we write the operator $\mathfrak{L}_n(\varepsilon,w)$ in \eqref{linop} as
\begin{equation}\label{explin}
\mathfrak{L}_n(\varepsilon,w)=\mathcal{L}_\omega h-\varepsilon P_n\Pi_W(b(t,x) h)-\varepsilon P_n\Pi_W(b(t,x) \partial_w v(w)[h]),
\end{equation}
where we denoted $b(t,x):=(\partial_u f)(w+v(w))=3(w+v(w))^2$. 

The operator 
$\mathcal{L}_{\omega}$ acts diagonally in the time-Fourier basis $\lbrace \exp(i\ell t) \rbrace_{\ell\in \Z}$, and it can be represented with the  operator valued matrix $\mathcal{L}_{\omega}=\operatorname{diag}(\omega^2 \ell^2 -A)_{\ell\in \Z}$, where $A$ is defined in \eqref{def.A}.
\begin{notation}\label{egenmenouno}
In the following we will use the convention $\egen_{-1}:=0$. Given $u\in \cH^0_x$ we denote $\langle u\rangle^\perp:=\left\lbrace \tilde{u} \in \cH^0_x\, :\, \langle u, \tilde{u}\rangle_{\cH^0_x}=0 \right\rbrace$.
\end{notation}
Recalling Definition \ref{anal.sp} we can expand functions in time in the basis of cosines or exponentials. For any $h\in W^{(n)}$ we have  
$$
b(t,x) h(t,x)=\sum\limits_{\ell_1\in \Z} \exp(i\ell_1 t)b_{\ell_1}(x)\sum\limits_{|k|\leq L_n} \exp(ik t)h_{k}(x)=\sum\limits_{\ell\in \Z,\, |k|\leq L_n}\exp(i\ell t) b_{\ell-k}(x)h_k(x) \, .
$$
We denote  $\pi_\ell:\cH^r_x\longmapsto \cH^r_x$ the $\cH^0_x$-orthogonal projector on $\langle\egen_{|\ell|-1}\rangle^\perp$. In view of Notation \ref{egenmenouno} $\pi_{0}=\mathds{1}_{\cH^0_x}$ and  it follows that
$$
P_n\Pi_W(b\,h)(t,x)=\sum\limits_{|\ell|, |k|\leq L_n}\exp(i\ell t) \pi_{\ell}(b_{\ell-k}(x)h_k(x)) \, .
$$
Thus, the linear operator $h\mapsto P_n\Pi_W(b\, h)$ can be represented in time-Fourier basis  
by the operator valued matrix $( \pi_{\ell}(b_{\ell-k}\, \cdot\,\ ))_{\ell,k} $ with indices  $ |\ell|,|k|\leq L_n$. In particular the diagonal entries (which correspond to the values $\ell=k$) are operators of the form $\pi_\ell (b_0\, \cdot\,)$, where $b_0(x)=\frac{1}{2\pi}\int_{0}^{2\pi}b(t,x) dt$ is the average in time of $b(t,x)$. 

It is convenient to split $\mathfrak{L}_n(\varepsilon,w)$ in \eqref{explin} as
\begin{equation}\label{linearizedsum}
\mathfrak{L}_n(\varepsilon,w) =D-\varepsilon\mathcal{M}_1-\varepsilon\mathcal{M}_2 \, ,
\end{equation}
where
\begin{equation}\label{Lpartsi}
D:=\mathcal{L}_\omega-\varepsilon P_n \Pi_n(b_0(x)\cdot)=\operatorname{diag}\left(\underbrace{\omega^2 \ell^2-A-\varepsilon \pi_{\ell}(b_0(x)\cdot)}_{=:D_\ell} \right)_{\ell=-L_n}^{L_n},
\end{equation}
and, denoting $\tilde{b}(t,x):=b(t,x)-b_0(x)$, 
\begin{equation}\label{Lpartsii}
\mathcal{M}_1:= P_n\Pi_W(\tilde{b}(t,x)\cdot)\,, \quad 
\mathcal{M}_2:=P_n\Pi_W\left(b(t,x)\partial_w v(w)[\cdot]\right)\,.
\end{equation}
The operators $D_\ell$ in \eqref{Lpartsi} act between $D_\ell: \cH^2_{x} \cap \langle \egen_{|\ell|-1}\rangle^\perp\longrightarrow \cH^0_{x} \cap \langle \egen_{|\ell|-1}\rangle^\perp$, and satisfy $D_{-\ell}=D_\ell$, for any $ \ell\in \Z $. So we will consider in the following the case $\ell\in \N$.

The operator $D$ is the main diagonal part of the operator $\mathfrak{L}_n (\varepsilon,w)$. We shall prove that   $\mathcal{M}_1$ is the main off-diagonal part, and that $\mathcal{M}_2$ is a remainder term since it contains $\partial_w v(w)$ which gains two derivatives with respect to $w$, see  Proposition \ref{soluzbif}. 
\subsection{Inversion of $D$}
In this section we prove that for $|\varepsilon|$ small enough, the operator  $D$ can be inverted. To this aim we start by proving that each $D_\ell$ can be diagonalized.
We first observe that $D_\ell$ is $\cH^0_x$-selfadjoint and $\langle \egen_{\ell-1}\rangle^{\perp}$-invariant.
We write $D_\ell=\omega^2\ell^2-S_\ell(\varepsilon)$ where
\begin{equation}\label{essekappa}
S_\ell(\varepsilon):=A+\varepsilon \pi_\ell b_0 \pi_\ell,
\end{equation}
and we introduce the bilinear forms
\begin{equation}\label{bilinear}
\langle u_1,u_2 \rangle_{i,\varepsilon}:= \langle S_\ell(\varepsilon)^i u_1,u_2 \rangle_{\cH^0_x} , \quad i=1,2 \, .
\end{equation}
Note that for $\varepsilon=0$ we have $\langle\, \cdot\, ,\, \cdot \, \rangle_{i,0}=\langle\, \cdot\, ,\, \cdot \, \rangle_{\cH^i_x}$.\\
For $|\varepsilon|\leq\frac{1}{2\|b_0\|_{\infty}}$, the bilinear form $\langle\, \cdot\,,\, \cdot\, \rangle_{1,\varepsilon}$ is an inner product on $\cH^1_x \cap  \langle \egen_{\ell-1}\rangle^\perp$, 
equivalent to $\langle\, \cdot\,, \,\cdot\, \rangle_{\cH^1_x} $, because
$$
(1-\varepsilon\|b_0\|_{\infty})\|u\|^2_{\cH^1_x}\leq \|u\|^2_{1,\varepsilon}\leq (1+\varepsilon\|b_0\|_{\infty})\|u\|^2_{\cH^1_x} \, .
$$
In the same way, for $|\varepsilon|\leq\frac{1}{4\|b_0\|_{\infty}}$, the bilinear form 
$\langle\, \cdot\,,\, \cdot\, \rangle_{2,\varepsilon}$ is an inner product on $\cH^2_x \cap  \langle \egen_{\ell-1}\rangle^\perp$, equivalent to $\langle\, \cdot\,, \,\cdot\, \rangle_{\cH^2_x}$, because 
\begin{equation}\label{normeq}
(1-2\varepsilon\|b_0\|_{\infty}-\varepsilon^2\|b_0\|^2_\infty )\|u\|^2_{\cH^2_x} \leq \|u\|^2_{2,\varepsilon} \leq (1+2\varepsilon\|b_0\|_{\infty}+\varepsilon^2\|b_0\|^2_\infty )\|u\|^2_{\cH^2_x}.
\end{equation}
In order to prove that $D_\ell$ can be diagonalized we use the following lemmas.

\begin{lemma}[Sobolev Embedding from the Sphere to the Circle]\label{sob}
For any $r\geq 0$ and $\delta>0$, the embedding 
$\cH^{r+1+\delta}_x \hookrightarrow H^{r}(\mathbb{S}^1,dx)$ is  compact. 
In particular there exists $ c(\delta)>0$ such that
\begin{equation}\label{sobolev}
\|u\|_{H^r(\mathbb{S}^1,d x)}\leq c(\delta) \|u\|_{\cH^{r+1+\delta}_x},\,\, \forall u \in \cH^{r+1+\delta}_x.
\end{equation}
\end{lemma}

\begin{proof}
We write the eigenfunctions $\egen_j(x)$ in \eqref{def.ej} as 
$$\egen_j(x)=\frac{\sin(\omega_j x)}{\sin x}=\frac{e^{i\omega_j x}-e^{-i\omega_j x}}{e^{i x}-e^{-ix}}=\frac{e^{-i\omega_j x}}{e^{-ix}}\frac{1-e^{i2x\cdot \omega_j}}{1-e^{2ix}}=e^{-ijx}\frac{1-(e^{2ix})^{\omega_j}}{1-e^{2ix}}.$$
Next, by the geometric identity $\sum\limits_{m=0}^n y^m=\frac{1-y^{n+1}}{1-y}$, we obtain
\begin{equation*}\egen_j(x)=e^{-ijx}\sum\limits_{m=0}^{j} e^{i2mx}=\sum\limits_{m=0}^j e^{i(2m-j)x}\,.
\end{equation*}
It follows that $\|\egen_j\|^2_{H^r(\mathbb{S}^1,dx)}=2\pi\sum\limits_{m=0}^{j}\langle 2m-j\rangle^{2r}\leq 2\pi \sum\limits_{m=0}^{j} \omega_j^{2r}=2\pi\omega_j^{2r+1}=2\pi\|\egen_j\|^2_{\cH^{r+\frac{1}{2}}_x}$.\\
Given $r\geq 0,$ $\delta>0$ and $ u=\sum\limits_{j\in\N}u_j\egen_j(x) \in \cH^{r+1+\delta}_x$ we obtain
\begin{align*}
&\|u\|_{H^r(\mathbb{S}^1,dx)}\leq \sum\limits_{j\in\N}|u_j|\|\egen_j\|_{H^r(\mathbb{S}^1,dx)}\leq \sqrt{2\pi}\sum\limits_{j\in\N} |u_j|\omega_j^{r+\frac{1}{2}}\\
&\leq \sqrt{2\pi}\left( \sum\limits_{j\in\N} |u_j|^2\omega_j^{2(r+\frac{1}{2}+\frac{1}{2}+\delta)}\right)^{\frac{1}{2}}\cdot \left(\sum\limits_{j\in\N} \omega_j^{-1-2\delta}\right)^{\frac{1}{2}}
\end{align*}
which proves \eqref{sobolev} with  $c(\delta) :=\sqrt{2\pi}\left(\sum\limits_{j\in\N} \omega_j^{-1-2\delta}\right)^{\frac{1}{2}}$.
\end{proof}

\begin{lemma}[Matrix elements of multiplication]\label{elements}
Let  $r\geq 0$, $\delta>0$ and 
let $b_0 (x) $ be a function in $ \cH^{r+1+\delta}_x$.
Then, for any $ j, k \in \N $, 
$$
|\langle b_0 \egen_j, \egen_k\rangle_{\cH^0_x}|\leq c(\delta)
\|b_0\|_{\cH_x^{r+1+\delta}}\left(\frac{1}{\langle k-j \rangle^{r}}+\frac{1}{|\omega_k+\omega_j|^{r}}\right).
$$
In particular, when $j=k$ 
\begin{equation}\label{diagonal}
\langle b_0\egen_j,\egen_j \rangle_{\cH^0_x}=\frac{1}{\pi}\int\limits_0^{\pi}b_0(x)dx +r_j(b_0),\quad \quad |r_j(b_0)|\leq c(\delta) \frac{\|b_0\|_{\cH_x^{r+1+\delta}}}{(2\omega_j)^r}.
\end{equation}
\end{lemma}
\begin{proof}
Using \eqref{equispazi} and the explicit representation of the eigenfunction \eqref{def.ej} we have
$$\langle b_0 \egen_j, \egen_k\rangle_{\cH^0_x}=\frac{2}{\pi}\int\limits_{0}^{\pi}b_0(x)\sin(\omega_jx)\sin(\omega_k x)dx.$$
Using the identity $\sin(\omega_jx)\sin(\omega_k x)=\frac{1}{2}\left(\cos(|k-j|x)-\cos((\omega_j+\omega_k)x)\right)$, and the inequality\\
$\left|\langle b_0,b_1\rangle_{L^{2}(\mathbb{S}^1,dx)}\right|\leq \|b_0\|_{H^r(\mathbb{S}^1)}\|b_1\|_{H^{-r}(\mathbb{S}^1)}$, with $b_1=\cos((\omega_j+\omega_k)x),\ \cos(|k-j|x)$ we obtain
$$|\langle b_0 \egen_j, \egen_k\rangle_{\cH^0_x}|\leq \|b_0\|_{H^r(\mathbb{S}^1,dx)}\left(\frac{1}{\langle k-j\rangle^{r}}+\frac{1}{(\omega_j+\omega_k)^r}\right)\stackrel{\eqref{sobolev}}\leq c(\delta)\|b_0\|_{\cH_x^{r+1+\delta}}\left(\frac{1}{\langle k-j\rangle^{r}}+\frac{1}{(\omega_j+\omega_k)^r}\right).
$$
In particular when $j=k$ we have $\sin(\omega_jx)^2=\frac{1}{2}-\frac{\cos(2\omega_jx)}{2}$ and so \eqref{diagonal} holds.
\end{proof}

\begin{definition}
For any $ \ell \in \N $, let define the Hilbert space
\begin{equation}\label{effelle}
F_\ell:= \cH^2_x\cap \langle\egen_{\ell-1}\rangle^{\bot}
\end{equation}
endowed with the inner product $\langle \cdot ,\cdot \rangle_{2,\varepsilon}$  in \eqref{bilinear}.
\end{definition}
\begin{proposition}[Sturm-Liouville] \label{sturm}
There exists $\varepsilon_0>0$ such that for any $ |\varepsilon|<\varepsilon_0$, there exist eigencouples $\lbrace \left(\phi_{\ell,j}(\varepsilon),  \lambda_{\ell,j}(\varepsilon)\right) \rbrace_{j\in \mathbb{N}, \omega_j \neq \ell}$, diagonalizing the operator $S_\ell(\varepsilon)$ defined in \eqref{essekappa}, and such that $\lbrace \phi_{\ell,j}(\varepsilon)\rbrace_{j\in\mathbb{N},\omega_j\neq \ell}$ are $F_\ell$-orthonormal and $\cH^0_x$-orthogonal. Furthermore, for any $ \delta \in ]0,1]$, there exists a constant $C:=C(\delta)$ such that the eigenvalues $\lambda_{\ell,j}(\varepsilon)$ satisfy
\begin{equation}\label{derlam}
\left|\lambda_{\ell,j}(\varepsilon)-\omega_j^2- \varepsilon\frac{1}{\pi}\int\limits_0^\pi b_0(x) dx\right|\leq C\frac{\varepsilon\|b_0\|_{\cH^2_x}}{\omega_j^{1-\delta}} \, .
\end{equation}
\end{proposition}

\begin{proof}
The operator $S_\ell(\varepsilon)$ is $\cH^0_x$-selfadjoint, and can be written as $S_\ell(\varepsilon)=A(\mathds{1}+\varepsilon A^{-1}\pi_\ell b_0 \pi_\ell)$. 
Since $\|A^{-1}\phi\|_{\cH^0_x}\leq \|\phi\|_{\cH^0_x}\,\, \forall \phi\in \cH^0_x$, we can invert $S_\ell(\varepsilon)$ for $|\varepsilon|< \frac{1}{\|b_0\|_\infty}$,
obtaining the convergent Neumann series 
\begin{equation}\label{inversoS}
K_\ell(\varepsilon):=S_\ell(\varepsilon)^{-1}=(\mathds{1}+\varepsilon A^{-1}\pi_\ell b_0\pi_\ell)^{-1}A^{-1}=\sum\limits_{p=0}^\infty(-\varepsilon A^{-1}\pi_\ell b_0 \pi_\ell)^{p}A^{-1} \, .
\end{equation}
The operator  $K_\ell(\varepsilon)$ is $\cH^0_x$-selfadjoint and compact, since $A^{-1}$ is compact (its eigenvalues are $\omega_j^{-2}$). Thus, by standard spectral theory, the operator
$S_\ell(\varepsilon)$ can be diagonalized with $\cH^0_x$-orthonormal eigenfuncitons $\left\lbrace \tilde{\phi}_{\ell,j}(\varepsilon)\right\rbrace_{\omega_j\neq \ell}$ with corresponding eigenvalues $\left\lbrace\lambda_{\ell,j}(\varepsilon)\right\rbrace_{\omega_j\neq \ell}$.\\
By rescaling $\phi_{\ell,j}(\varepsilon):=\left(\lambda_{\ell,j}(\varepsilon)\right)^{-1}\tilde{\phi}_{\ell,j}(\varepsilon)$, we obtain a $\cH^0_x$-orthogonal and $\langle \cdot, \cdot \rangle_{2,\varepsilon}$-orthonormal basis. 
By classical spectral theory we also know that for $|\varepsilon|$ small enough,  the eigenvalues of $S_\ell(\varepsilon)$ are still simple since $A$ has simple spectrum, moreover the eigenvalues $\lambda_{\ell,j}(\varepsilon)$ and the eigenfunctions $\tilde{\phi}_{\ell,j}(\varepsilon)$ are smooth in $\varepsilon$ and the following formulas hold:
\begin{equation}\label{autoval}
\begin{aligned}
&\lambda_{\ell,j}(\varepsilon)=\langle S_\ell(\varepsilon) \tilde{\phi}_{\ell,j}(\varepsilon),\tilde{\phi}_{\ell,j}(\varepsilon)\rangle_{\cH^0_x} \, ,\\
&\partial_\varepsilon\lambda_{\ell,j}(\varepsilon)=\langle \pi_\ell b_0 \pi_\ell\tilde{\phi}_{\ell,j}(\varepsilon),\tilde{\phi}_{\ell,j}(\varepsilon)\rangle_{\cH^0_x} \, ,\\
&\partial_\varepsilon^2\lambda_{\ell,j}(\varepsilon)=2\langle \pi_\ell b_0 \pi_\ell\tilde{\phi}_{\ell,j}(\varepsilon),\partial_\varepsilon\tilde{\phi}_{\ell,j}(\varepsilon)\rangle_{\cH^0_x} \, .
\end{aligned}
\end{equation}
Evaluating at $\varepsilon=0$ we have $\partial_\varepsilon\lambda_{\ell,j}(0)=\langle b_0\egen_j,\egen_j\rangle_{\cH^0_x}$ which we estimated in \eqref{diagonal}.
Thus, since by Taylor remainder formula we have
\begin{equation}\label{taylor}
\lambda_{\ell,j}(\varepsilon)=\omega_j^2+\varepsilon (\partial_{\varepsilon}\lambda_{\ell,j})(0)+\int\limits_0^\varepsilon\int\limits_0^\nu (\partial_\varepsilon^2\lambda_{\ell,j})(\tilde{\nu})d\tilde{\nu}d\nu,
\end{equation}
we are left to estimate the second derivative term in \eqref{taylor} in order to prove \eqref{derlam}. According to \eqref{autoval} we need to estimate $\|\partial_\varepsilon\tilde{\phi}_{\ell,j}(\varepsilon)\|_{\cH^0_x}$ in order to obtain a bound for $\partial_\varepsilon^2\lambda_{\ell,j}(\varepsilon)$.\\
We denote as $P_\varepsilon$ the $\cH^0_x$-orthogonal projector on $\langle\tilde{\phi}_{\ell,j}(\varepsilon)\rangle^\perp$, and $\mu_{\ell,j}(\varepsilon):=\lambda_{\ell,j}(\varepsilon)^{-1}$.\\
We differentiate the identity $ (K_\ell(\varepsilon)-\mu_{\ell,j}(\varepsilon))\tilde{\phi}_{\ell,j}(\varepsilon)=0$, where $K_\ell(\varepsilon)$ is defined in \eqref{inversoS}, and we obtain
\begin{equation}\label{acaso2}
(K_\ell(\varepsilon)-\mu_{\ell,j}(\varepsilon))\partial_\varepsilon\tilde{\phi}_{\ell,j}(\varepsilon)=-(\partial_\varepsilon K_\ell(\varepsilon)-\partial_\varepsilon\mu_{\ell,j}(\varepsilon))\tilde{\phi}_{\ell,j}(\varepsilon) \, .
\end{equation}
Then we apply the projector $P_\varepsilon$, which commutes with $K_\ell(\varepsilon)-\mu_{\ell,j}(\varepsilon)$, to both sides of \eqref{acaso2}, and using $\langle \partial_\varepsilon\tilde{\phi}_{\ell,j}(\varepsilon),\tilde{\phi}_{\ell,j}(\varepsilon)\rangle_{\cH^0_x}=0$ (which is equivalent to $P_\varepsilon\partial_\varepsilon\tilde{\phi}_{\ell,j}=\partial_\varepsilon\tilde{\phi}_{\ell,j}(\varepsilon)$) and that $K_\ell(\varepsilon)-\mu_{\ell,j}(\varepsilon)$ is invertible on $\operatorname{Rg}(P_\varepsilon)$ we obtain
$$
\partial_\varepsilon \tilde{\phi}_{\ell,j}(\varepsilon)=-(K_\ell(\varepsilon)-\mu_{\ell,j}(\varepsilon))^{-1}P_\varepsilon (\partial_\varepsilon K_\ell)(\varepsilon)\tilde{\phi}_{\ell,j}(\varepsilon) \, .
$$
Since $K_\ell(\varepsilon)=S_{\ell}(\varepsilon)^{-1}$, we have that $\partial_\varepsilon K_\ell(\varepsilon)=-K_\ell(\varepsilon)\partial_\varepsilon S_\ell(\varepsilon)K_\ell(\varepsilon)$, 
and it follows
\begin{equation}\label{derauto}
\partial_\varepsilon\tilde{\phi}_{\ell,j}(\varepsilon)=\mu_{\ell,j}(\varepsilon)(K_\ell(\varepsilon)-\mu_{\ell,j}(\varepsilon))^{-1}P_\varepsilon K_\ell(\varepsilon)\pi_\ell b_0 \pi_\ell
 \tilde{\phi}_{\ell,j}(\varepsilon)=T\pi_\ell b_0\pi_\ell \tilde{\phi}_{\ell,j}(\varepsilon),
\end{equation}
where we denoted $T:=\mu_{\ell,j}(\varepsilon)(K_\ell(\varepsilon)-\mu_{\ell,j}(\varepsilon))^{-1}P_\varepsilon K_\ell(\varepsilon)$.\\
For any $v:=\sum\limits_{k \neq j}\hat{v}_k \tilde{\phi}_{\ell,k}(\varepsilon)$ we have
$$
Tv=\sum\limits_{k \neq j} \frac{\hat{v}_k}{\lambda_{\ell,j}(\varepsilon)-\lambda_{\ell,k}(\varepsilon)}\tilde{\phi}_{\ell,k}(\varepsilon) \, .
$$
Recalling \eqref{autoval}, we see that $|\partial_\varepsilon \lambda_{\ell,j}(\varepsilon)|\leq \|b_0\|_{\infty}$, and so for $\varepsilon\leq \frac{1}{4\|b_0\|_\infty}$ we have:
\begin{align*}
&|\lambda_{\ell,k}(\varepsilon)-\lambda_{\ell,j}(\varepsilon)|\geq |\lambda_{\ell,k}(0)-\lambda_{\ell,j}(0)|-|\lambda_{\ell,k}(\varepsilon)-\lambda_{\ell,k}(0)|-|\lambda_{\ell,j}(\varepsilon)-\lambda_{\ell,j}(0)|\\
&\geq |\omega_k^2-\omega_j^2|-2\varepsilon \|b_0\|_\infty\geq \frac{1}{2}|\omega_k+\omega_j|,\quad\quad \forall k\neq j \, .
\end{align*}
Hence we have $\|T\|_{B(\cH^0_x)}\leq \frac{2}{\omega_j}$, and by \eqref{derauto} we obtain
$$
\|\partial_\varepsilon \tilde{\phi}_{\ell,j}(\varepsilon)\|_{\cH^0_x}\leq \|T(\pi_\ell b_0 \pi_\ell \tilde{\phi}_{\ell,j}(\varepsilon))\|_{\cH^0_x}\leq \frac{2}{\omega_j}\|b_0\|_\infty \, .
$$
Plugging this inequality into \eqref{autoval} we obtain
\begin{equation}\label{dersec}
|\partial_\varepsilon^2\lambda_{\ell,j}(\varepsilon)|\leq \frac{4\|b_0\|^2_\infty}{\omega_j}.
\end{equation}
Finally, we plug \eqref{diagonal} and \eqref{dersec} into \eqref{taylor} and we obtain
$$\left|\lambda_{\ell,j}(\varepsilon)-\omega_j^2-\varepsilon\frac{1}{\pi}\int\limits_0^\pi b_0(x)dx\right|\leq 4\varepsilon^2\frac{\|b_0\|^2_\infty}{\omega_j}+ c(\delta)\varepsilon \frac{\|b_0\|_{\cH^2_x}}{\omega_j^{1-\delta}}\leq 2c(\delta)\varepsilon \frac{\|b_0\|_{\cH^2_x}}{\omega_j^{1-\delta}} \, ,
$$
which proves \eqref{derlam} with  $C:=2c(\delta)$.
\end{proof}
%
In order to prove invertibility of $D_\ell$ we prove the following estimates of the small divisors.
\begin{lemma}\label{alfakappa}For any $\gamma \in \left]0,\frac{1}{6}\right[$, and $\tau \in ]1,2[$, there exists $\varepsilon_0 :=\varepsilon_0(\gamma,\tau)>0$ small enough such that, for any $\varepsilon \in G_n(w)\subseteq [0,\varepsilon_0]$ defined in \eqref{cantorAn}, and $\omega$ as in \eqref{amplfreq} then
\begin{equation}\label{alfa}
\alpha_\ell:=\min\limits_{\stackrel{j\in\N}{\omega_j\neq \ell}} |\omega^2\ell^2-\lambda_{\ell,j}(\varepsilon)|\geq \frac{\gamma}{ 20\langle \ell\rangle ^{\tau-1}} \, , \ \forall \ell\in \Z \, .
\end{equation}
\end{lemma}

\begin{proof}
From \eqref{derlam} it follows that for any $ \ell\in \Z,\, |\ell| \neq \omega_j$, 
\begin{equation}\label{serviva}
|\omega^2\ell^2-\lambda_{\ell,j}(\varepsilon)|\geq \left| \omega^2\ell^2-\omega_j^2 -\varepsilon M(w)\right|-C(\delta)\frac{\varepsilon \|b_0\|_{\cH^2_x}}{\omega_j^{1-\delta}},
\end{equation}
where, according to  \eqref{cantor}, $M(w)=\frac{1}{\pi}\int_0^\pi b_0(x)dx$, and $b_0(x)=\frac{1}{2\pi}\int_{0}^{2\pi}3(w(t,x)+v(w)(t,x))^2dt.$\\
For $\ell=0$ we have $\alpha_0=\min\limits_{j\in \N}|\lambda_{0,j}|=\lambda_{0,0}(\varepsilon)$. Since $\lambda_{0,0}(0)=1$, then \eqref{alfa} holds for any $\varepsilon$ small enough by continuity. 
For $\ell\neq 0$ we first note that, by simmetry of $\omega^2\ell^2$ and $\lambda_{\ell,j}$, we have $\alpha_\ell=\alpha_{-\ell}$, so we consider $\ell\geq 1$. We shall look at 2 different cases:
\\[1mm]
{\sc $1^{st}$ case: $0<\ell\leq\frac{1}{3\varepsilon}.$} Since $|\omega-1|\leq 2 \varepsilon$ and $ \ell\neq \omega_j$, then we have:
$$
|\omega \ell-\omega_j|=|\ell-\omega_j+(\omega-1)\ell|\geq |\ell-\omega_j|-|\omega-1||\ell|\geq 1-2\varepsilon\frac{1}{3\varepsilon}=\frac{1}{3} \, .
$$
It follows by \eqref{serviva} $|\omega^2 \ell^2-\lambda_{\ell,j}|\geq \frac{1}{3}-O(\varepsilon)\geq \frac{1}{4}\geq \frac{\gamma}{20|\ell|^{\tau-1}}$.
\\[1mm]
{\sc $2^{nd}$ case:  $\ell>\frac{1}{3\varepsilon}$.}
Applying \eqref{serviva} for some $\delta \in ]0,2-\tau]\,$ we obtain
\begin{align*}
&|\omega^2\ell^2-\lambda_{\ell,j}(\varepsilon)|\geq \left|\omega^2\ell^2-\omega_j^2-\varepsilon M(w)\right|-C(\delta)\varepsilon\frac{\|b_0\|_{\cH^2_x}}{\omega_j^{1-\delta}}\\
&=\left|\omega\ell-\sqrt{\omega_j^2+\varepsilon M(w)}\right|\left|\omega\ell+\sqrt{\omega_j^2+\varepsilon M(w)}\right|-C(\delta)\|b_0\|_{\cH^2_x}\varepsilon\omega_j^{-1+\delta}.
\end{align*}
By Taylor $\sqrt{\omega_j^2+\varepsilon M(w)}=\omega_j \left(1+\varepsilon\frac{M(w)}{2\omega_j^2}+O\left(\frac{\varepsilon^2}{\omega_j^4}\right)\right)$, it follows:
\begin{align*}
&|\omega^2\ell^2-\lambda_{\ell,j}(\varepsilon)|\geq \omega \ell\cdot \left| \omega \ell-\omega_j-\varepsilon\frac{M(w)}{2\omega_j}\right|-\hat{C}(\delta)\left(\frac{\varepsilon}{\omega_j^{1-\delta}}+\frac{\varepsilon^2 \ell}{\omega_j^3}\right)\\
&\underbrace{\geq}_{\varepsilon \in G_n(w)}\frac{\omega \ell\gamma}{|\ell+\omega_j|^{\tau}}-\hat{C}(\delta)\left(\frac{\varepsilon}{\omega_j^{1-\delta}}+\frac{\varepsilon^2\ell}{\omega_j^3}\right)
\end{align*}
for some positive constant $\hat{C}(\delta)>0$ which tends to $+\infty$ as $\delta$ tends to $0$. For any $\tau \in ]1,2[$ we take $\delta\in ]0,2-\tau]$ so that $\tau-1 \leq 1-\delta$.\\
Now, if $|\omega \ell -\omega_j|\geq 1$, then by \eqref{serviva} $|\omega^2\ell^2-\lambda_{\ell,j}(\varepsilon)|\geq \omega_j-C\varepsilon\geq \frac{\omega_j}{2}$, and so the thesis holds. 
Hence, we can suppose that $\alpha_\ell$ is realized for some $j=j(\ell)$ such that $|\omega \ell- \omega_j|<1$, which implies $\omega_j\leq \omega \ell +1\leq 3\ell$ and so $\ell+\omega_j\leq 4\ell$. 
It follows $\omega \ell\cdot \frac{\gamma}{|\ell+\omega_j|^{\tau}}\geq \ell\cdot\frac{\gamma}{(4\ell)^\tau}\geq \frac{\gamma}{16 \ell^{\tau-1}}$. 
On the other hand, we also have $\omega \ell\leq \omega_j +1$, so $\ell\leq \omega_j+1\leq 2\omega_j$, hence $\omega_j\geq \frac{\ell}{2}$, and so, provided $\varepsilon\gamma^{-1}\ll \hat{C}(\delta)^{-1}$ we have
$$\hat{C}(\delta)\left(\frac{\varepsilon}{\omega_j^{1-\delta}}+\frac{\varepsilon^2 \ell}{\omega_j^3}\right)\leq 4\hat{C}(\delta) \frac{\varepsilon}{\ell^{1-\delta}}\leq  \frac{\gamma}{80\ell^{\tau-1}}.$$
It follows finally $|\omega^2\ell^2-\lambda_{\ell,j}(\varepsilon)|\geq \frac{\gamma}{16\ell^{\tau-1}}-\frac{\gamma}{80\ell^{\tau-1}}\geq\frac{\gamma}{20
\ell^{\tau-1}}$, which gives our statement.
\end{proof}
We define $|D_\ell|^{-\frac{1}{2}}$ as the linear operator acting as $|D_\ell|^{-\frac{1}{2}}\phi_{\ell,j}(\varepsilon)=|\omega^2\ell^2-\lambda_{\ell,j}(\varepsilon)|^{-\frac{1}{2}}\phi_{\ell,j}(\varepsilon)$, where $\phi_{\ell,j}(\varepsilon)$ is the $F_\ell$-orthonormal basis defined in Proposition \ref{sturm}. We also define $|D|^{-\frac{1}{2}}:=\text{diag}(|D_\ell|^{-\frac{1}{2}})_{\ell\in\Z}$.
\begin{lemma}[Invertibility of $D$]\label{Di}
For any $\varepsilon \in G_n(w)\subseteq [0,\varepsilon_0]$, with $\varepsilon_0$ given in Lemma \ref{alfakappa}, the operator $D_\ell$ is invertible in $F_\ell$ (defined in \eqref{effelle}, and
\begin{equation}\label{stimadk}
\||D_\ell|^{-\frac{1}{2}}h\|_{2,\varepsilon}\leq \frac{1}{\sqrt{\alpha_\ell}}\|h\|_{2,\varepsilon} \, ,\quad \forall h \in F_\ell 	\, .
\end{equation}
In addition the operator $|D|^{\frac{1}{2}}$ is invertible on $W^{(n)}$ for any $n\in\N$ and it satisfies
\begin{equation}\label{estdiag}
\left \| |D|^{-\frac{1}{2}}h\right\|_{\sigma,s}\leq \frac{9}{\sqrt{\gamma}}\|h\|_{\sigma, s+\frac{\tau-1}{2}},\quad \forall h \in W^{(n)}.
\end{equation}
\end{lemma}

\begin{proof}
By Lemma \ref{alfakappa} $|\omega^2\ell^2-\lambda_{\ell,j}|\geq \alpha_\ell>0$, for any $ \ell\neq \omega_j \in \N $ and therefore  each $D_\ell^{-1}$ is well-defined and \eqref{stimadk} holds.\\
Moreover by \eqref{alfa}, \eqref{stimadk} and the equivalence of the norms $\|\cdot \|_{\cH^2_x},\, \|\cdot\|_{2,\varepsilon}$ in \eqref{normeq}, we have that
for any $ h\in W^{(n)} $, 
$|D|^{-\frac{1}{2}}h=\sum\limits_{|\ell|\leq L_n}\exp(i\ell t)|D_\ell|^{-\frac{1}{2}}h_\ell$ satisfies   
\begin{align*}
\left\| |D|^{-\frac{1}{2}}h\right\|^2_{\sigma,s}= \sum\limits_{|\ell|\leq L_n}\exp(2\sigma|\ell|)\langle \ell \rangle^{2s}\left\| |D_\ell|^{-\frac{1}{2}}h_\ell\right\|^2_{\cH^2_x}\leq \sum\limits_{|\ell|\leq L_n}\exp(2\sigma|\ell|)\langle \ell \rangle^{2s}\frac{4}{\alpha_\ell}\|h_\ell\|^2_{\cH^2_x}\\
\leq 80\sum\limits_{|\ell|\leq L_n}\exp(2\sigma|\ell|)\langle \ell \rangle^{2s}\frac{\langle \ell\rangle^{\tau-1}}{\gamma}\left\|h_\ell\right\|^2_{\cH^2_x}\leq 80\gamma^{-1}\|h\|^2_{\sigma, s+\frac{\tau-1}{2}}.
\end{align*}
This proves \eqref{estdiag}. 
\end{proof}

\subsection{Invertibility of $\mathfrak{L}_n $}

In the following we show how to recover the invertibility of $\mathfrak{L}_n := \mathfrak{L}_n (\varepsilon, w)$ defined in \eqref{linearizedsum}  from the invertibility of its main diagonal part $D$.
It is convenient to factor out $\mathfrak{L}_n(\varepsilon,w)$ by $|D|^{\frac{1}{2}}$ writing 
$$
\mathfrak{L}_n= |D|^{\frac{1}{2}}\left( U-\varepsilon \mathcal{R}_1-\varepsilon \mathcal{R}_2\right)|D|^{\frac{1}{2}},
$$
where, recalling \eqref{Lpartsi},\eqref{Lpartsii},  the operators 
$U$, $\mathcal{R}_1$ and $\mathcal{R}_2$ are  
\begin{equation}\label{Lparts2}
U:=\textrm{sgn}(D) := \big(\textrm{sgn}(D_\ell))_{\ell \in \Z} \big), 
\quad \mathcal{R}_i :=|D|^{-\frac{1}{2}}\mathcal{M}_i |D|^{-\frac{1}{2}},\, i=1,2 \, .
\end{equation}
\begin{lemma}[Estimate of $U$]\label{U}
There exists $\varepsilon_0>0$ small enough such that, for any $ |\varepsilon |\leq\varepsilon_0$, the operator $ U $ in \eqref{Lparts2} satisfies 
\begin{equation}\label{stimaU}
\|Uh\|_{\sigma, s} \leq 4\|h\|_{\sigma,s} \, , \quad \forall h \in W^{(n)}, \,\,\, \quad \forall \sigma \geq 0,\, \forall s>\frac{1}{2}\,. 
\end{equation}
\end{lemma}

\begin{proof}
By definition $\forall h \in W^{(n)},\,\textrm{sgn}(D)h=\sum\limits_{|\ell|\leq L_n}\exp(i\ell t)\textrm{sgn}(D_\ell )h_\ell$. Moreover the norms $\|\cdot\|_{2,\varepsilon}$ and $\|\cdot\|_{\cH^2_x}$ are equivalent for $|\varepsilon|\leq \varepsilon_0$ small enough by \eqref{normeq}. It follows that
\begin{align*}
&\|Uh\|^2_{\sigma,s}=\sum\limits_{|\ell|\leq L_n}\exp(2\sigma|\ell|)\langle \ell \rangle^{2s}\|\textrm{sgn}(D_\ell) h_\ell\|^2_{\cH^2_x}\\
&\leq 2\sum\limits_{|\ell|\leq L_n}\exp(2\sigma|\ell|)\langle \ell \rangle^{2s}\|\textrm{sgn}(D_\ell) h_\ell\|^2_{2,\varepsilon}
\leq 4\sum\limits_{|\ell|\leq L_n}\exp(2\sigma|\ell|)\langle \ell \rangle^{2s}\| h_\ell\|^2_{\cH^2_x}\leq 4\|h\|_{\sigma,s}^2
\end{align*}
proving the lemma.
\end{proof}

\begin{lemma}[Analysis of Small Divisors]\label{smalldiv}
Let $\sigma\geq 0,\, s>\frac{1}{2}$, and let $\beta:=\frac{2-\tau}{\tau}\in ]0,1[$. Then 
there exist $\bar{C}>0$ and $ \varepsilon_0 := \varepsilon_0(\sigma,s, \gamma,\tau)>0$ small enough such that, for any $\varepsilon\in G_n(w)\subseteq [0,\varepsilon_0]$, the following estimates hold
\begin{equation}\label{piccdiv}
\frac{1}{\alpha_\ell \alpha_k}\leq \bar{C} \frac{|k-\ell|^{2\frac{\tau-1}{\beta}}}{\gamma^2\varepsilon^{\tau-1}},\quad \forall \ell, k\in \Z,\,\, \ell \neq k \, .
\end{equation}
\end{lemma}

\begin{proof}
{\sc First case:} sites far away from the diagonal, namely $|k-\ell|\geq \frac{1}{2}\max \lbrace |k|,|\ell|\rbrace^{\beta
}$. By \eqref{alfa} we have for $\varepsilon \leq 2^{-\frac{2}{\beta}}$
$$(\alpha_k\alpha_\ell)^{-1}\leq (20\gamma^{-1})^2\max\lbrace |k|,|\ell|\rbrace^{2(\tau-1)}\leq (20\gamma^{-1})^2(2|k-\ell|)^\frac{2(\tau-1)}{\beta}\leq C_1\gamma^{-2}\varepsilon^{-(\tau-1)}|k-\ell|^{\frac{2(\tau-1)}{\beta}}.$$

The other cases consist of sites close to the diagonal, namely when $|k-\ell|<\frac{1}{2}\max\lbrace |k|, |\ell| \rbrace^{\beta}$.\\We point out that in the case of sites close to the diagonal $\textrm{sgn}(k)=\textrm{sgn}(\ell)$, in fact if we assume $\textrm{sgn}(k)\neq \textrm{sgn}(\ell)$ then $|k-\ell|=|k|+|\ell|>\max\lbrace |k|, |\ell| \rbrace>\frac{1}{2}\max\lbrace |k|, |\ell| \rbrace^{\beta}$, since $\beta\in ]0,1[.$ Moreover by simmetry $\alpha_{-k}=\alpha_k$, so from now on we assume $k,\ell\geq 0$. We also point out that for sites which are close to the diagonal $k$ and $\ell$ have the same size, namely: \begin{equation}\label{ellekappa}
\frac{|k|}{2}\leq |\ell|\leq 2|k|,\,\,\,\text{and}\,\,\,\frac{|\ell|}{2}\leq |k|\leq 2|\ell| \, .
\end{equation}
In fact if $k\geq\ell$ then $|k-\ell|\leq\frac{|k|^\beta}{2}\leq \frac{|k|}{2}$, hence we have $|k|\leq |k-\ell|+|\ell|\leq |\ell|+\frac{|k|}{2}$ and so $ \frac{|k|}{2}\leq |\ell|$.\\
If $\ell \geq k$ we obtain with the same considerations that $\frac{|\ell|}{2}\leq |k|\leq |\ell|$.\\
{\sc Second case:} $|k-\ell|\leq \frac{\max \lbrace k,\ell\rbrace^{\beta}}{2}$, $k \leq \frac{1}{3\varepsilon}\, \vee\, \ell \leq \frac{1}{3\varepsilon}$.\\
If $k\leq \frac{1}{3\varepsilon}$, then for $\varepsilon$ small enough we have $\alpha_k\geq \frac{k+1}{12}$, in fact $|\omega k-\omega_j|\geq |k-\omega_j|-|\omega-1||k|>1-2\varepsilon\frac{1}{3\varepsilon}\geq \frac{1}{3}$. Then $|\omega^2k^2-\omega_j^2|\geq \frac{k}{3}$ and by \eqref{serviva} it follows
\begin{align*}
&|\omega^2k^2-\lambda_{k,j}(\varepsilon)|\geq |\omega^2k^2-\omega_j^2|-O(\varepsilon)
\geq\frac{k}{3}-C\varepsilon
\geq \frac{k+1}{6}-\frac{1}{6}\geq \frac{k+1}{12}.
\end{align*}
Now if also $\ell<\frac{1}{3\varepsilon}$, then $(\alpha_k \alpha_\ell)^{-1}\leq \frac{12^2}{(k+1)(\ell+1)}\leq12^2$, while if $\ell\geq\frac{1}{3\varepsilon}$ we use the bound given by \eqref{alfa} and \eqref{ellekappa} obtaining $(\alpha_k \alpha_\ell)^{-1}\leq \frac{12}{k+1}\frac{20|\ell|^{\tau-1}}{\gamma}\leq \frac{C_2}{\gamma}$ for some constant $C_2>0$.\\
We are left to consider when $|k-\ell|\leq \frac{\max\lbrace \ell, k \rbrace^\beta}{2},\, k>\frac{1}{3\varepsilon},\, \ell >\frac{1}{3\varepsilon}$. We define the integer numbers $j:=j(k)=\textrm{arg}\min\limits_{\stackrel{j'\in \N,}{\omega_{j'}\neq k}}|\omega^2k^2-\lambda_{k,j'}(\varepsilon)|,$ and in the same way we define $i:=i(\ell)$.\\
{\sc Third case:} $|k-\ell|\leq \frac{\max\lbrace \ell, k \rbrace^\beta}{2},\, k>\frac{1}{3\varepsilon},\, \ell >\frac{1}{3\varepsilon},\,\, k-\omega_j=\ell-\omega_i$.
\begin{align*}
&|(\omega k -\omega_j)-(\omega \ell -\omega_i)|=|\omega(k-\ell)-(\omega_j-\omega_i)|=
|\omega-1||k-\ell|\geq \frac{\varepsilon}{2}\\ &\Rightarrow |\omega k -\omega_j|\geq \frac{\varepsilon}{4}\, \vee \, |\omega \ell -\omega_i|\geq \frac{\varepsilon}{4}.
\end{align*}
Assume $|\omega k -\omega_j|\geq \frac{\varepsilon}{4}$, then $|\omega^2k^2-\omega_j^2|\geq \frac{\varepsilon}{4}\omega k \geq \frac{1}{12}$. It follows $\alpha_k\geq  \frac{1}{24}$ since  $\alpha_k\geq |\omega^2k^2-\omega_j^2|-C\varepsilon$.\\
Then, by \eqref{alfa} we obtain for some constants $C_3>c>0$ that $(\alpha_k\alpha_\ell)^{-1}\leq c\frac{\ell^{\tau-1}}{\gamma }\leq  C_3\frac{1}{\gamma \varepsilon^{\tau-1}}$.\\
If $|\omega_\ell-\omega_i|\geq \frac{\varepsilon}{4}$ we proceed as before with $k,\ell$ having inverted roles, obtaining the same bound.\\
{\sc Fourth case:} $|k-\ell|\leq \frac{\max \lbrace k, \ell \rbrace^\beta}{2},\, k>\frac{1}{3\varepsilon},\, \ell >\frac{1}{3\varepsilon},\,\, k-\omega_j\neq \ell-\omega_i$.\\
Using the condition $\varepsilon\in G_n(w)$ and recalling $\omega_j=j+1$, we obtain:
\begin{align*}
&|(\omega k -\omega_j)-(\omega \ell-\omega_i)|=|\omega(k-\ell)-(j-i)|
\geq \frac{\gamma}{|k-\ell|^{\tau}}\geq \frac{2^\tau \gamma}{\max\lbrace k,\ell\rbrace^{\beta \tau}}=\frac{2^\tau \gamma}{\max\lbrace k,\ell\rbrace^{2- \tau}},
\end{align*}
 and so 
$$|\omega k-\omega_j|\geq \frac{2^{\tau-1}\gamma}{\max\lbrace k,\ell\rbrace^{2- \tau}}\, \vee \, |\omega \ell-\omega_i|\geq \frac{2^{\tau-1}\gamma}{\max\lbrace k,\ell\rbrace^{2- \tau}}.$$
If the first one holds then for some $c',c''>0$ we have $|\omega^2k^2-\omega_j^2|\geq \frac{c'\gamma}{k^{2-\tau}}\omega k \geq c''\gamma k^{\tau-1}$, and so $\alpha_k \geq \frac{c''}{2}\gamma k^{\tau-1}$. It follows by \eqref{alfa} that for $C_4>40c''$ we have $(\alpha_k\alpha_\ell)^{-1}\leq 20c'' \frac{1}{\gamma k^{\tau-1}}\frac{\ell^{\tau-1}}{\gamma}\leq C_4\frac{1}{\gamma^2}$.\\
If the second one holds then the conclusion is the same by inverting the roles of $k$ and $\ell$.

In conclusion,  the previous cases finally imply that 
$$
(\alpha_k\alpha_\ell)^{-1}\leq \max\left\lbrace C_1\frac{|k-\ell|^\frac{2(\tau-1)}{\beta}}{\gamma^2\varepsilon^{\tau-1}},\, \frac{C_2}{\gamma},\,\frac{C_3}{\gamma\varepsilon^{\tau-1}},\,\frac{C_4}{\gamma^2} \right\rbrace\leq \bar{C}\frac{|k-\ell|^{\frac{2(\tau-1)}{\beta}}}{\gamma^2\varepsilon^{\tau-1}}
$$
where $\bar{C} :=\max\lbrace C_1,C_2,C_3,C_4\rbrace$. 
\end{proof}

\begin{lemma}[Estimate of $\mathcal{R}_1$]\label{r1}
Assume $\|w\|_{\sigma, s+\frac{2\tau(\tau-1)}{2-\tau}}<\rho$ defined in \eqref{rho}. Then, for any $\varepsilon\in G_n(w) $,  the operator $\mathcal{R}_1$ defined in \eqref{Lparts2} is bounded in $ X_{\sigma, s+\frac{\tau-1}{2}}$, in particular there exists $\tilde{C}=\tilde{C}(\sigma,s)>0$ such that
\begin{equation}\label{stimar1}
\|\mathcal{R}_1h\|_{\sigma, s+\frac{\tau-1}{2}}\leq \frac{\bar{C}}{\gamma \varepsilon^{\frac{\tau-1}{2}}}\|h\|_{\sigma,s+ \frac{\tau-1}{2}},\quad\forall h \in W^{(n)}.
\end{equation}
\end{lemma}

\begin{proof}
Since $\|w\|_{\sigma,s+\frac{2\tau(\tau-1)}{2-\tau}+2}< \rho$, then $v(w)\in X_{\sigma,s+\frac{2\tau(\tau-1)}{2-\tau}+2}$ by Proposition \ref{soluzbif}. Moreover, recalling $b=2(v+w(v))^2$, by algebra estimate \eqref{algebra} there exists $R'>0$ such that 
\begin{equation}\label{bbb}
\|b\|_{\sigma,s+\frac{2\tau(\tau-1)}{2-\tau}}\leq R'.
\end{equation}
Now recalling  \eqref{Lparts2}, expanding 
$h(t,x)=\sum\limits_{|k|\leq L_n}\exp(ikt)h_k(x) \in W^{(n)}$, we have (cfr \eqref{Lpartsii})
$$
|D|^{-\frac{1}{2}}P_n\Pi_W(\tilde{b}|D|^{-\frac{1}{2}}h)=\sum\limits_{|\ell|\leq L_n}\exp(i\ell t)|D_\ell|^{-\frac{1}{2}}\pi_\ell\left(\sum\limits_{|k|\leq L_n}b_{\ell-k}\delta_{\ell \neq k}|D_k|^{-\frac{1}{2}}h_k\right).
$$
Thus we obtain
\begin{equation}\label{mormaR}
\|\mathcal{R}_1 h\|^2_{\sigma,s+\frac{\tau-1}{2}}=\sum\limits_{|\ell|\leq L_n}\exp(2\sigma|\ell|)\langle \ell\rangle^{2s+\tau-1}\left\| |D_\ell|^{-\frac{1}{2}}\pi_\ell\sum\limits_{\stackrel{|k|\leq L_n}{k\neq \ell}}b_{\ell-k}|D_k|^{-\frac{1}{2}}h_k\right\|^2_{\cH^2_x}.
\end{equation}
Let us denote $\mathtt{B}_0:=0,\,	\mathtt{B}_m:=\|b_m\|_{\cH^{2}_x},$ if $m\in \Z\setminus \lbrace 0 \rbrace$, then by \eqref{stimadk},\eqref{piccdiv} 
\begin{align*}
&\left\| |D_\ell|^{-\frac{1}{2}}\pi_\ell\sum\limits_{\stackrel{|k|\leq L_n}{k\neq \ell}}b_{\ell-k}|D_k|^{-\frac{1}{2}}h_k\right\|_{\cH^2_x}\leq \frac{1}{\sqrt{\alpha_\ell}}\sum\limits_{\stackrel{|k|\leq L_n}{k\neq \ell}}\mathtt{B}_{\ell-k} \frac{1}{\sqrt{\alpha_k}}\|h_k\|_{\cH^2_x}\\
&\leq \bar{C} \sum\limits_{\stackrel{|k|\leq L_n}{k\neq \ell}}\frac{|k-\ell|^{\frac{\tau-1}{\beta}}}{\gamma \varepsilon^{\frac{\tau-1}{2}}}\mathtt{B}_{k-\ell}\|h_k\|_{\cH^2_x}\leq \bar{C} \gamma^{-1}\varepsilon^{-\frac{\tau-1}{2}}S_\ell\quad \quad \textrm{where}\quad  S_\ell:=\sum\limits_{\stackrel{|k|\leq L_n}{k\neq \ell}}|k-\ell|^{\frac{\tau-1}{\beta}}\mathtt{B}_{k-\ell}\|h_k\|_{\cH^2_x}.
\end{align*}
Now let $\tilde{S}(t):=\sum\limits_{|\ell|\leq L_n}\exp(i\ell t)S_\ell$, then by \eqref{mormaR} we have
\begin{equation}\label{serviva22}
\|\mathcal{R}_1h\|_{\sigma,s+\frac{\tau-1}{2}}\leq \bar{C}\gamma^{-1}\varepsilon^{-\frac{\tau-1}{2}}\|\tilde{S}\|_{\sigma,s+\frac{\tau-1}{2}}.
\end{equation}
Writing $\tilde{S}(t)=P_n(\mathtt{B}\cdot \mathtt{h})$, where $\mathtt{B}(t)=\sum\limits_{k\in \Z\setminus\lbrace 0\rbrace}\exp(ikt)|k|^{\frac{\tau-1}{\beta}}\mathtt{B}_k$, and $\mathtt{h}(t):=\sum\limits_{|k|\leq L_n}\exp(ikt)\|h_k\|_{\cH^2_x}$, using algebra estimates \eqref{algebra} we have (remember $\beta=\frac{2-\tau}{\tau}$):
\begin{equation}\label{stimaesse}
\begin{aligned}
&\|\tilde{S}\|_{\sigma, s+\frac{\tau-1}{2}}\leq C(s) \|\mathtt{B}\|_{\sigma,s+\frac{\tau-1}{2}}\|\mathtt{h}\|_{\sigma,s+\frac{\tau-1}{2}}\\
&\leq C(s)\|b\|_{\sigma,s+(\tau-1)\left(\frac{1}{2}+\frac{1}{\beta}\right)}\|h\|_{\sigma, s+\frac{\tau-1}{2}}\underbrace{\leq}_{\frac{1}{2}+\frac{1}{\beta}\leq \frac{2}{\beta},\eqref{bbb}} C(s)R'\|h\|_{\sigma,s+\frac{\tau-1}{2}}.
\end{aligned}
\end{equation}
By plugging \eqref{stimaesse} into \eqref{serviva22}, we obtain the thesis with $\tilde{C}=C(s)R'\bar{C}$.
\end{proof}
\begin{lemma}[Estimate of $\mathcal{R}_2$]\label{r2}
Assume $\|w\|_{\sigma, s+\frac{2\tau(\tau-1)}{2-\tau}}<\rho$ defined in \eqref{rho}. Then, for any $\varepsilon\in G_n(w) $, the operator $\mathcal{R}_2$ defined in \eqref{Lparts2} is bounded in $ X_{\sigma, s+\frac{\tau-1}{2}}$, in particular there exists $\tilde{C}=\tilde{C}(\sigma,s)>0$ such that
\begin{equation}\label{stimar2}
\|\mathcal{R}_2h\|_{\sigma, s+\frac{\tau-1}{2}}\leq \tilde{C}\gamma^{-1}\|h\|_{\sigma,s+ \frac{\tau-1}{2}} \, ,\quad\forall h \in W^{(n)} \, .
\end{equation}
\end{lemma}
\begin{proof}
Recalling $\tau \in ]1,2[$ and \eqref{Lpartsii}, \eqref{Lparts2}, we have by algebra estimate \eqref{algebra}, \eqref{estdiag} and \eqref{bbb}
\begin{align*}
&\|\mathcal{R}_2 h \|_{\sigma,s+\frac{\tau-1}{2}}\leq \frac{9}{\sqrt{\gamma}}\left\|P_n\Pi_W\left(b\cdot \partial_w v(w)\left[|D|^{-\frac{1}{2}}h\right]\right)\right\|_{\sigma, s+\tau-1}\\
&\leq C(s)\frac{1}{\sqrt{\gamma}}\|b\|_{\sigma, s+\tau-1}\left\| \partial_w v(w)\left[|D|^{-\frac{1}{2}}h\right]\right\|_{\sigma, s+\tau-1} \leq C(s) \frac{R'}{\sqrt{\gamma}}\left\|\partial_w v(w)\left[|D|^{-\frac{1}{2}}h\right]\right\|_{\sigma, s+2}\\
&\underbrace{\leq}_{\text{by Proposition } \ref{soluzbif}} C(s)\frac{RR'}{\sqrt{\gamma}}\left\| |D|^{-\frac{1}{2}} h\right\|_{\sigma, s}\leq \frac{\tilde{C}(\sigma,s)}{\gamma}\|h\|_{\sigma, s+\frac{\tau-1}{2}}.
\end{align*}
This proves \eqref{stimar2}. 
\end{proof}

\begin{proof}[Proof of Proposition \ref{inversolinearizzato}]
In view of \eqref{Lparts2}, $U=U^{-1}$,  in particular we can write $U-\varepsilon\mathcal{R}_1-\varepsilon\mathcal{R}_2=U(\mathds{1}-\varepsilon U\mathcal{R}_1-\varepsilon U\mathcal{R}_2)$. Now, $U$ is invertible on $X_{\sigma,s+\frac{\tau-1}{2}}$ and satisfies \eqref{stimaU} and we can invert $\mathds{1}-\varepsilon U\mathcal{R}_1- \varepsilon U\mathcal{R}_2 $ in $X_{\sigma,s+\frac{\tau-1}{2}}$ for $\varepsilon<\frac{1}{2}\left(\frac{\gamma}{8\tilde{C}}\right)^2$ by Lemmata \ref{r1}, \ref{r2} using Neumann series. We deduce that $U-\varepsilon \mathcal{R}_1-\varepsilon\mathcal{R}_2$ is invertible and $\|(U-\varepsilon\mathcal{R}_1-\varepsilon\mathcal{R}_2)^{-1}h\|_{\sigma,s+\frac{\tau-1}{2}}\leq 8\|h\|_{\sigma,s+\frac{\tau-1}{2}}$, $\forall h\in W^{(n)}$. Recalling \eqref{estdiag} and letting
$K=8\cdot 9^2$, we conclude that, for any $ h \in W^{(n)} $,
$$
\|\mathfrak{L}_n(\varepsilon,w)^{-1}h\|_{\sigma,s}=\left\| |D|^{-\frac{1}{2}}(U-\varepsilon\mathcal{R}_1-\varepsilon\mathcal{R}_2)^{-1}|D|^{-\frac{1}{2}}h\right\|_{\sigma,s}\leq \frac{K}{\gamma}\|h\|_{\sigma,s+\tau-1} \, . 
$$
In particular $\|\mathfrak{L}_n(\varepsilon,w)^{-1}h\|_{\sigma,s}\leq \frac{K}{\gamma} L_n^{\tau-1}\|h\|_{\sigma,s} $ for any $ h \in W^{(n)}$.
\end{proof}

\footnotesize 

\end{document}